\documentclass[a4paper,12pt]{article}
\usepackage{latexsym,amsmath,amsthm,amssymb}
\usepackage{a4wide}
\usepackage{hyperref}
\usepackage{marginnote}
\usepackage{color}
\usepackage{cite}
\hypersetup{
pdftitle={Stability GNS}   
pdfauthor={Van Hoang Nguyen},
colorlinks = true,
linkcolor = magenta,
citecolor = blue,
}

\theoremstyle{plain}
\newtheorem{theorem}{Theorem}[section]

\newtheorem{proposition}[theorem]{Proposition}
\newtheorem{lemma}[theorem]{Lemma}
\newtheorem{corollary}[theorem]{Corollary}

\theoremstyle{definition}

\newtheorem{remark}[theorem]{Remark}

\renewcommand{\thefootnote}{\arabic{footnote}}
\def\R{\mathbb R}% tap so thuc
\def\om{\omega}% omega
\def\be{\beta}% beta
\def\de{\delta}% delta
\def\vphi{\varphi}% varphi
\def\na{\nabla}% nabla
\def\lt{\left}% trai
\def\rt{\right}% phai
\def\o{\overline}

\def\i0i{\int_0^\infty}

\def\deGNS{\delta_{\rm GNS}}
\def\hatde{\hat{\de}_{\rm GNS}}
\def\lamGNS{\lambda_{\rm GNS}}

\numberwithin{equation}{section}

\title{The sharp Gagliardo--Nirenberg--Sobolev inequality in quantitative form}
\author{Van Hoang Nguyen\footnote{
Institut de Math\'ematiques de Toulouse, Universit\'e Paul Sabatier, 118 Route de Narbonne, 31062 Toulouse c\'edex 09, France.}
}

\begin{document}
\maketitle

%% Classification and key words; note that the 2010 classification is used:

\renewcommand{\thefootnote}{}

\footnote{Email: \href{mailto: Van Hoang Nguyen <van-hoang.nguyen@math.univ-toulouse.fr>}{van-hoang.nguyen@math.univ-toulouse.fr}}

\footnote{2010 \emph{Mathematics Subject Classification\text}: 26D10.}

\footnote{\emph{Key words and phrases\text}: Gagliardo--Nirenberg--Sobolev inequality, weighted Sobolev inequality, stability estimates, sharp constant.}

\renewcommand{\thefootnote}{\arabic{footnote}}
\setcounter{footnote}{0}

\begin{abstract}
Using a dimension reduction argument and a stability version of the weighted Sobolev inequality on half space recently proved by Seuffert, we establish, in this paper, some stability estimates (or quantitative estimates) for a family of the sharp Gagliardo--Nirenberg--Sobolev inequalities due to Del Pino and Dolbeault \cite{DelPino}.
\end{abstract}

\section{Introduction}
%It is well--known that the Sobolev type inequalities have many applications in many branches of mathematics such as analysis, differential geometry, differential partial equations, calculus of variations, etc. Rough speaking, such inequalities ensure a higher regularity of a function when its distributional gradient enjoys some integrability. 
%Throughout this paper, let $W^{1,p}(\R^n)$ denote the space of measurable function on $\R^n, n\geq 2$ that have a $L^p-$integrable distributional gradient. For $1\leq p < n$, the Sobolev inequality asserts the existence of the constant $S_{n,p}$ depending only on $n,p$ such that
The Gagliardo--Nirenberg--Sobolev (GNS) inequality in $\R^n$ with $n\geq 2$ asserts the existence of a positive constant $C$ such that
\begin{equation}\label{eq:GNS}
\|u\|_r \leq C\, \|\na u\|_p^{\theta} \, \|u\|_q^{1-\theta},
\end{equation}
where $p,q,r,\theta$ are parameters satisfying the conditions 
\[
1< p< n,\quad 1\leq q \leq r \leq p^*,\quad p^* = \frac{np}{n-p},
\]
and
\begin{equation*}
\frac1r = \frac{\theta}{p^*} + \frac{1-\theta}q,
\end{equation*}
and where $u$ is taken in $\mathcal D^{p,q}(\R^n)$ which is the completion of $C_0^\infty(\R^n)$ under the norm $\|u\|_{\mathcal D^{p,q}} = \|u\|_q + \|\na u\|_p$. Using variational argument and P\'olya--Szegg\"o principle \cite{Brothers}, we can show that the extremal functions for \eqref{eq:GNS} exist and are determined uniquely by a positive, decreasing and spherical symmetric function up to a multiple by a constant, to a translation and to a dilation. However, the explicit formula for the extremal functions and for the best constant in \eqref{eq:GNS} is still unknown except some special choice of parameters $p,q,r$. For example, when $q=r =p^*$ and $\theta =1$, \eqref{eq:GNS} reduces to the Sobolev inequality which the sharp constant and the set of extremal functions were found independently by Aubin \cite{Aubin} and Talenti \cite{Talenti} (see \cite{Rosen} for an earlier result in $\R^3$). For $p=r =2$, $q=1$, and $\theta = n/(n+2)$,\eqref{eq:GNS} reduces to the Nash inequality which the sharp constant $C_n$ was found by Carlen and Loss in \cite{CarlenLoss}.

Del Pino and Dolbeault \cite{DelPino} found the best constant and classified all extremal functions of the GNS inequality  for a special one parameter family of parameters $p,q,r$ with $p=2$, $q = t+1$ and $r =2t$ for $1<t < n/(n-2)$. More precisely, they proved the following inequality 
\begin{equation}\label{eq:DD}
\|u\|_{2t} \leq A_{n,t} \|\na u\|_2^{\theta}\, \|u\|_{p+1}^{1-\theta},\quad \theta = \frac{n(t-1)}{t[2n -(1+t)(n-2)]},
\end{equation}
for any $u\in \mathcal D^{2,t+1}(\R^n)$, with
\[
A_{n,t} = \lt(\frac{y(t-1)^2}{2\pi n}\rt)^{\theta/2} \lt(\frac{2y -n}{2y}\right)^{1/2t} \lt(\frac{\Gamma(y)}{\Gamma(y -n/2)}\rt)^{\theta/n},\qquad y = \frac{t+1}{t-1}, 
\]
and equality holds in \eqref{eq:DD} if and only if
\[
u(x) = c \lt(1 + |a(x-x_0)|^2\rt)^{-1/(t-1)},
\]
for some $c\in \R$, $a >0$ and $x_0\in \R^n$. When $t$ decreases to $1$, \eqref{eq:DD} reduces to an optimal Euclidean logarithmic Sobolev inequality which is equivalent to the famous logarithmic Sobolev inequality for Gaussian measure of Gross \cite{Gross}. The sharp GNS inequality for another one parameter family of parameters $p =2$, $q = 2t$ and $r =t+1$ with $0< t< 1$ also were proved in \cite{DelPino}. In \cite{DelPino03}, Del Pino and Dolbeault generalized their result in \cite{DelPino} to any $p\in (1,n)$. Another proofs of the results of Del Pino and Dolbeault and its generalization to any norm of gradient (not need Euclidean norm) were given in \cite{Bakry,Cordero,Nguyen15}.

Recent years, the problem of finding stability estimates for the sharp inequalities both in analysis and geometry such as isoperimetric inequality, Brunn--Minkowski inequality, Sobolev inequality, logarithmic Sobolev inequality, etc, were intensively studied. For example, the stability version of the Euclidean isoperimetric inequality was proved in \cite{Bonnesen,Figalliiso,Fuglede,Fusco08,Hall91,Hall,Maggi} while the quantitative form of the Brunn--Minkowski inequality was established in \cite{FigalliBMconvex,Figallisumset,FigalliBM}. We refer the reader to \cite{Barchiesi,Cianchigauss,Eldan} for the stability versions of the isoperimetric inequality in Gaussian spac and to \cite{CianchiPS,Fusco09} for the quantitative form of the P\'olya--Szeg\"o principle and of the Faber--Krahn type inequalities. The stability estimates for the Sobolev inequality in the bounded domains were first proved by Brezis and Lieb in \cite{Brezis}. Since the paper of Brezis and Lieb, there are many works on the stability form of the Sobolev inequality. For example, Bianchi and Egnell \cite{Bianchi} established a stability version for the $L^2-$Sobolev inequality in whole space $\R^n$ which answers affirmatively a question of Brezis and Lieb in \cite{Brezis}. The quantitative form of the $L^p-$Sobolev inequality with $p\not=2$ was proved by Cianchi, Fusco, Maggi and Pratelli \cite{Cianchi} and recently by Figalli and Neumayer \cite{Figalli}. The stability version of the Sobolev inequality on functions of bounded variation were studied by Cianchi \cite{CianchiBV}, Fusco, Maggi and Pratelli \cite{Fusco07}, and by Figalli, Maggi and Pratelli \cite{FigalliBV}. See also \cite{Chen,Lu,Loiudice} for the stability version of the other Sobolev type inequality (higher order and on Heisenberg group), and see \cite{Dolbeault11,Dolbeault13,Jankowiak} for the other improvement of the $L^2-$Sobolev inequality with the remainder involving to the Hardy--Littlewood--Sobolev inequality. The stability results for the logarithmic Sobolev inequality can be found in \cite{Bobkov,Dolbeault16,Indrei,Fathi}.

%The Bianchi--Egnell argument was then applied to find the stability version of various sharp inequalities, e.g., the higher order Sobolev inequality \cite{Chen,Lu}, the sharp Sobolev inequality in Heisenberg groups \cite{Loiudice} and the sharp $L^p-$Sobolev inequality with $p\geq 2$ (see \cite{Figalli}), the weighted Sobolev inequality in half space \cite{Seuffert1}, etc. The quantitative form of the sharp Sobolev inequality \eqref{eq:Sobolev} for all $p\not=2$ was recently found bi Cianchi, Fusco, Maggi and Pratelli \cite{Cianchi} by using the sharp quantitative form of the isoperimetric inequality in \cite{Fusco08} and a reduction argument to the $n-$symmetric functions. The sharp Sobolev inequality in the case $p=1$ is extended to all functions of bounded variation. Its stability version was established by Cianchi in \cite{CianchiBV}, by Fusco, Maggi, and Pratelli in \cite{Fusco07}. The sharp quantitative form for the $L^1-$Sobolev inequality was recently found by Figalli, Maggi, Pratelli in \cite{FigalliBV} by using mass transportation method.  can be found in .

Contrary with the Sobolev inequality, a few stability version for the GNS inequality is known, e.g., \cite{Carlen13,Carlen14,Dolbeault13,Dolbeault16,Nguyen16,Ruffini,Seuffert2}. The fact that the GNS inequality involving three not two norms (as Sobolev inequality) makes difficulties to establish their stability version. This fact prevents any direct adaption of the proof of Bianchi and Egnell \cite{Bianchi} to any of the other cases of the GNS inequality for which the optimizers are known. Also, the proof based on the optimal transportation of measures \cite{Figalliiso} and on the symmetrization techniques \cite{Cianchi,FigalliBV,Fusco07,Fusco08} did not procedure any results in this situation. The first stability results for the GNS inequality were established by Carlen and Figalli \cite{Carlen13} and by Dolbeault and Toscani \cite{Dolbeault13}. In their interesting paper \cite{Carlen13}, Carlen and Figalli exploited a stability result of Bianchi and Egnell for the Sobolev inequality in $\R^4$ and a dimension reduction  argument introduced by Bakry \cite{Bakry} to establish some stability estimates for a special GNS inequality in $\mathcal D^{2,4}(\R^2)$ and then applied them to obtain the explicit convergence rate to equilibrium for the critical mass Keller--Segel equation and the stability estimate for the logarithmic Hardy--Littlewood--Sobolev inequality. They also mentioned in their paper that their method can be used to obtain the stability results for whole family of GNS inequality \eqref{eq:DD}. This was completely done in recent work of Seuffert \cite{Seuffert2} by using the technique of Carlen and Figalli and his stability version for the weighted Sobolev inequality on half space \cite{Seuffert1}. For $1< t < n/(n-2)$, denote $2(t) = 2(4t+n -nt)/(n+2 +2t -nt)$ and 
\begin{equation}\label{eq:deltamu}
\hatde[u] = A_{n,t}^{4t/2(t)} \|\na u\|_2^{\theta 4t/2(t)} \|u\|_{t+1}^{(1-\theta)4t/2(t)} - \|u\|_{2t}^{4t/2(t)}, \quad u\in \mathcal D^{2,t+1}(\R^n).
\end{equation}
Throughout this paper, for $a >0$ and $x_0\in \R^n$, we define
\[
v_{a,x_0}(x) = (1 + a^2 |x-x_0|^2)^{-1/(t-1)},
\]
and denote $v_{1,0}$ by $v$ for simplicity. It was proved by Seuffert that there exist positive constants $K_1$ and $\de_1$ depending only on $n$ and $t$ such that for any nonnegative function $u\in \mathcal D^{2,t+1}(\R^n)$ such that $\|u\|_{2t} = \|v\|_{2t}$ and $\hatde[u] \leq \de_1$, then
\begin{equation}\label{eq:Seuffert}
\inf_{a >0, x_0} \|u^{2t} -a^n v_{a,x_0}^{2t}\|_1 \leq K_1 \hatde[u]^{1/2}.
\end{equation}
When $t=3, n=2$, \eqref{eq:Seuffert} goes back to the result of Carlen and Figalli (Theorem $1.2$) in \cite{Carlen13}. The improved version of \eqref{eq:DD} (in the nonhomogeneous form) was established in \cite{Dolbeault13} by Dolbeault and Toscani using the nonlinear evolution equations (fast diffusion) and improved entropy--entropy product estimates. In \cite{Dolbeault16}, these improvements were reproved by a simple proof (by the same authors) and were applied to give a faster convergence of solutions toward the equilibrium in the porous medium equations. In \cite{Nguyen16}, the author gives another proof for the result of Dolbeault and Toscani using mass transportation method, and extend it for any $1< p< n$ and for any norm of gradient (not need Euclidean norm). In \cite{Carlen14}, Carlen, Frank and Lieb proved a stability result for a GNS inequality which does not belong to the family \eqref{eq:DD} by means of Bianchi and Egnell method. This result then is applied to give the stability estimates for the lowest eigenvalue of a Schr\"odinger operator.

Our aim in this paper is to provide the stability estimates for the GNS inequality \eqref{eq:DD}. To do this, let us introduce the \emph{GNS deficit functional\text} on $\mathcal D^{2,t+1}(\R^n)$ by
\begin{equation}
\de_{\rm GNS}[u] = \frac{A_{n,t} \|\na u\|_2^{\theta} \, \|u\|_{t+1}^{1-\theta}}{ \|u\|_{2t}} -1, \quad u \in \mathcal D^{2,t+1}(\R^n),
\end{equation}
if $u\not\equiv 0$ and $\deGNS[0] =0$. We also introduce the concept of asymmetry following Ruffini \cite{Ruffini} by
\begin{equation}\label{eq:asymmetry}
\lambda_{\rm GNS}[u] = \inf_{a >0, x_0\in \R^n} \Bigg{\|} \frac{\|v\|_{2t}}{\|u\|_{2t}} \, u - a^{\frac n{2t}}v_{a,x_0} \Bigg{\|}_{2t}^{2t},
\end{equation}
if $u\not\equiv 0$, and $\lamGNS[0] =0$. By \eqref{eq:DD}, $\deGNS[u] >0$ unless $u$ is a multiple of $v_{\lambda,x_0}$ for some $\lambda >0$ and some $x_0\in \R^n$. 
%The main question in this paper is addressed as follows: \emph{How close is $u$ to some multiple of $v_{\lambda,x_0}$ for some $\lambda >0$ and some $x_0\in \R^n$ when $\deGNS[u]$ is small?\text}

Our first result in this paper is the following.
\begin{theorem}\label{1stTheorem}
Let $n\geq 2$ and $1 < t < (2n+1)/(2n-3)$. Let $u\in \mathcal D^{2,t+1}(\R^n)$ be a nonnegative function such that 
\begin{equation}\label{eq:conditiononu}
\|u\|_{2t} = \|v\|_{2t},\qquad\text{ and }\qquad \frac{t^2-1}{2n}\|\na u\|_2^2 = \|u\|_{t+1}^{t+1}.
\end{equation}
Then there exists constants $K$ and $\delta$ depending only on $n$ and $t$ such that whenever $\hatde[u] \leq \de$,
\begin{equation}\label{eq:1ststability}
\inf_{x_0\in \R^n} \lt(\int_{\R^n} |\na u -\na v_{1,x_0}|^2 dx + \int_{\R^n} \lt|u^{\frac{t+1}2} -v_{1,x_0}^{\frac{t+1}2}\rt|^2 dx \rt)\leq K \hatde[u].
\end{equation}
\end{theorem}
The restriction $1< t< (2n+1)/(2n-3)$ comes from the fact that the dimension reduction argument does not implies the full family of GNS inequality \eqref{eq:DD} as mentioned in \cite{Nguyen15}. Since the functional $\deGNS$ is invariant under the change of function $u$ to $\lambda^{n/2t} u(\lambda \cdot)$ for $\lambda >0$, hence we always can choose a $\lambda >0$ such that the second condition in \eqref{eq:conditiononu} holds. Comparing with the result of Carlen and Figalli, and of Seuffert, we see that \eqref{eq:1ststability} gives us a lower bound of $\deGNS[u]$ in terms of $\|\na u -\na v_{1,x_0}\|_2^2$ and of $\|u^{(t+1)/2} -v_{1,x_0}^{(t+1)/2}\|_2^2$ for some $x_0\in \R^n$. We will show that \eqref{eq:1ststability} actually implies \eqref{eq:Seuffert}. 

Another consequence of Theorem \ref{1stTheorem} is the following quantitative form of \eqref{eq:DD}.
\begin{corollary}\label{quantitativeform}
Let $n\geq 2$ and $1< t< (2n+1)/(2n-3)$. There exists a constant $C >0$ depending only on $n,t$ such that for any $u\in \mathcal D^{2,t+1}(\R^n)$, the following estimate
\begin{equation}\label{eq:quantiativeform}
\lamGNS[u]^{\frac{t+1}{t[2(1-\theta) + (t+1)\theta]}} \leq C \deGNS[u],
\end{equation}
holds.
\end{corollary}

The power $(t+1)/[t(2(1-\theta) + (t+1)\theta)]$ of $\lamGNS[u]$ in \eqref{eq:quantiativeform} is not sharp. Its sharp value should be $1/t$. We next prove a similar result for the density $u^{t+1}$. To do so, we need to require additional some priori bounds ensuring some uniform integrability of the class of densities satisfying the bounds. It is natural to use moment bounds and entropy bounds (as done in \cite{Carlen13}).

Define
\begin{equation}\label{eq:NSdefine}
N_p(u) = \int_{\R^n} |y|^p u^{t+1}(y) dy,\qquad S(u) = \int_{\R^n} u^{t+1} \ln(u^{t+1}) dy.
\end{equation}
\begin{theorem}\label{2ndTheorem}
Let $n\geq 2$ and $1< t< (2n+1)/(2n-3)$. Let $u\in \mathcal D^{2,t+1}(\R^n)$ be a nonnegative function such that $\|u\|_{t+1} = \|v\|_{t+1}$. Suppose that for some $A, B < \infty$ and $1< p < 2(t+1)/(t-1) -n$, 
\begin{equation}\label{eq:bounds}
S(u) \leq A, \qquad N_p(u) \leq B,
\end{equation}
and assume also that 
\begin{equation}\label{eq:centercond}
\int_{\R^n} x \, u(x)^{t+1} dx = 0.
\end{equation}
Then there exists constants $K_2, \delta_2$ depending only on $n,t,p, A$ and $B$  such that whenever $\hatde[u] \leq \de_2$,
\begin{equation}\label{eq:2ndstability}
\inf_{a > 0} \|u^{t+1} -a^{n} v_{a,0}^{t+1}\|_1 \leq K_2 \, \hatde[u]^{(p-1)/(2p)}.
\end{equation}
\end{theorem}
Note that $v(x) \sim |x|^{-2/(t-1)}$ then the condition $p< 
2(t+1)/(t-1) -n$ is rather natural for the finite of $N_p(u)$. The case $n=2, t=3$, Theorem \ref{2ndTheorem} is exactly Theorem $1.4$ of Carlen and Figalli in \cite{Carlen13}. However, the order of $\hatde[u]$ in our Theorem \ref{2ndTheorem} is better than the one in Theorem $1.4$ of Carlen and Figalli which value is $(p-1)/(4p)$. As an application of our improvement in Theorem \ref{2ndTheorem}, we can improve the stability result for the Log--HLS inequality and the convergence rate to equilibrium for the solution of the Keller--Segel equation established by Carlen and Figalli in \cite{Carlen13} (at least twice).

Let us explain how to prove these results. Our method used in this paper is the modification of the one given by Carlen and Figalli \cite{Carlen13} and Seuffert \cite{Seuffert2}. We combining the stability version of the weighted Sobolev inequality established in \cite{Seuffert1} by Seuffert and the dimension reduction argument of Bakry \cite{Bakry} to obtain Theorem \ref{1stTheorem}. The main different between our proof and the one of Carlen and Figalli, and of Seuffert is that after applying the stability version of the weighted Sobolev inequality on the half space, we do not apply the weighted Sobolev inequality to the remainder term. Instead of this, we make some computations to control the remainder term when the deficit is small. Theorem \ref{1stTheorem} then follows by the special form of the functions which we define on the half space.

One of the main ingredients in our proof is the stability version of the weighted Sobolev inequality on half space $\R^{n+1}$ due to Seuffert \cite{Seuffert1}. The sharp weighted Sobolev inequality on half space was proved by the author in \cite{Nguyen15} by means of the mass transportation technique which generalizes one result of Bakry, Gentil and Ledoux in \cite{Bakry}. By adapting a dimension reduction argument due to Bakry, the author derived a subfamily of GNS inequality due to Del Pino and Dolbeault \cite{DelPino,DelPino03} (for any $1< p < n$ and even for any norm of gradient) from the weighted Sobolev inequality on half space. Let $1< t < (2n+1)/(2n-3)$, denote 
\begin{equation}\label{eq:fractionaldimension}
s= \frac{2n+1 -(2n-3)t}{t-1},\quad n_s = n+s+1,\quad\text{and}\quad 2_s^* = \frac{2n_s}{n_s-2}.
\end{equation}  
It was proved by the author in \cite{Nguyen15} (see also \cite{Bakry}) that the following inequality
\begin{equation}\label{eq:weightedSobolev}
\lt(\int_{\R_+^{n+1}} |f(x,y)|^{2_s^*} y^s dxdy\rt)^{\frac 2{2_s^*}} \leq S_{n,s} \int_{\R_+^{n+1}} |\na f(x,y)|^2 y^s dx dy,
\end{equation}
holds with the sharp constant $S_{n,s}$ (its explicit value can be found in \cite{Nguyen15}), and the equality holds if and only if 
\begin{equation}\label{eq:exfunction}
f(x,y) = c\, a^{\frac{n_s-2}2}\lt(1 + a^2 |x-x_0|^2 + a^2 y^2\rt)^{-\frac{n_s-2}2} =: g_{c,a,x_0}(x,y)
\end{equation}
for some $c\in \R$, $a >0$ and $x_0\in \R^n$. In \cite{Seuffert1}, by adapting the proof of Bianchi and Egnell, Seuffert established a stability version of \eqref{eq:weightedSobolev} as follows
\begin{multline}\label{eq:stabweighted}
S_{n,s} \int_{\R_+^{n+1}} |\na f(x,y)|^2 y^s dx dy - \lt(\int_{\R_+^{n+1}} |f(x,y)|^{2_s^*} y^s dxdy\rt)^{\frac 2{2_s^*}} \\
\geq C \inf_{c\in \R, a >0, x_0\in \R^n} \int_{\R^{n+1}_+} \lt|\na\lt(u(x,y) -g_{c,a,x_0}(x,y)\rt)\rt|^2 y^s dx dy,
\end{multline}
for some constant $C$ depending only on $n$ and $t$. The inequality \eqref{eq:stabweighted} plays an important role in the work of Seffert \cite{Seuffert2} and in our work in this paper.

The rest of this paper is organized as follows. The next section \S2 is devoted to prove the stability results in Theorem \ref{1stTheorem} and Corollary \ref{quantitativeform}. We also show how Theorem \ref{1stTheorem} implies the results of Carlen and Figalli and of Seuffert in this section. The proof of Theorem \ref{2ndTheorem} is given in section \S3.

\section{Proof of Theorem \ref{1stTheorem}}
\subsection{From weighted Sobolev to GNS inequality}
We begin by explaining the argument deriving the GNS inequality from the weighted Sobolev inequality on the half space \cite{Bakry,Nguyen16}. Let $u \in \mathcal D^{2,t+1}(\R^n)$ be a nonnegative function satisfying \eqref{eq:conditiononu}, we define a new function on $\R^{n+1}_+$ by
\begin{equation}\label{eq:defineoff}
f(x,y) = (u(x)^{1-t} + y^2)^{-\frac{n_s-2}2}.
\end{equation}
For $b > -1$, $a > 0$ and $2a-b >1$, denote
\[
D(a,b) = \int_0^\infty (1+r^2)^{-a} r^b dr = \frac12 \frac{\Gamma((b+1)/2) \Gamma((2a-b-1)/2)}{\Gamma(a)}.
\]
Note that 
\begin{equation}\label{eq:a*}
D(a,b+2) = \frac{b+1}{2a-b-3} D(a,b).
\end{equation}
Then we have the following result.

\begin{proposition}\label{bridge}
Let $u \in \mathcal D^{2,t+1}(\R^n)$ be a nonnegative function satisfyting \eqref{eq:conditiononu}. Suppose that $f$ is defined as \eqref{eq:defineoff}, then we have
\begin{align}\label{eq:bridge}
D(n_s,s)^{\frac2{2(t)}}\hatde[u]=S_{n,s}\int\limits_{\R^{n+1}_+} |\na f(x,y)|^2 y^s dx dy - \lt(\int\limits_{\R^{n+1}_+} f(x,y)^{2_s^*} y^s dx dy\rt)^{\frac 2{2_s^*}}.
\end{align}
\end{proposition}
Proposition \ref{bridge} provides a bridge between the GNS inequality \eqref{eq:DD} and the sharp weighted Sobolev inequality on the half space \eqref{eq:weightedSobolev}. Our interest in this proposition is that it relates the GNS deficit to the Sobolev deficit. We will give a quick proof of this proposition below.
\begin{proof}
It is easy to check that
\begin{equation}\label{eq:canbang}
\frac{2 \theta t}{2(t)} + \frac{4(1-\theta) t}{(1+t) 2(t)} = 1.
\end{equation}
By a suitable change of variable, we get
\begin{equation*}
\int_{\R^{n+1}_+} f(x,y)^{2_s^*} y^s dx dy = D(n_s,s) \int_{\R^n} u(x)^{2t} dx,
\end{equation*}
and
\begin{multline*}
\int_{\R^{n+1}_+} |\na f(x,y)|^2 y^s dx dy = \lt(\frac{(t-1)(n_s-2)}2\rt)^2 D(n_s,s) \int_{\R^n} |\na u(x)|^2 dx \\
+ (n_s-2)^2 D(n_s,s+2) \int_{\R^n} u(x)^{t+1} dx.
\end{multline*}
Hence, a straightforward computation shows that
\begin{align*}
&S_{n,s}\int_{\R^{n+1}_+} |\na f(x,y)|^2 y^s dx dy - \lt(\int_{\R^{n+1}_+} f(x,y)^{2_s^*} y^s dx dy\rt)^{\frac 2{2_s^*}} \notag\\
&=S_{n,s} (n_s-2)^2 \lt[\frac{(t-1)^2}4 D(n_s,s) \|\na u\|_2^2
+ D(n_s,s+2) \|u\|_{t+1}^{t+1}\rt]\notag -D(n_s,s)^{\frac{2}{2(t)}} \|u\|_{2t}^{\frac{4t}{2(t)}}\notag\\
&= S_{n,s} (n_s-2)^2 \lt[\frac{n(t-1)}{2(t+1)} D(n_s,s) + D(n_s,s+2)\rt] \|u\|_{t+1}^{t+1} - D(n_s,s)^{\frac{2}{2(t)}} \|u\|_{2t}^{\frac{4t}{2(t)}}\quad\text{\rm (by \eqref{eq:conditiononu})}\notag\\
&=S_{n,s}(n_s-2)^2 D(n_s,s) \frac{n-nt+4t}{2(t+1)} \|u\|_{t+1}^{t+1} -D(n_s,s)^{\frac{2}{2(t)}} \|u\|_{2t}^{\frac{4t}{2(t)}}\hspace{3.1cm} \text{\rm (by \eqref{eq:a*})}\notag\\
&= S_{n,s}(n_s-2)^2 D(n_s,s) \frac{n-nt+4t}{2(t+1)} \lt(\frac{t^2-1}{2n}\rt)^{\frac{2\theta t}{2(t)}} \|\na u\|_2^{\frac{4\theta t}{2(t)}} \|u\|_{t+1}^{\frac{(1-\theta)4t}{2(t)}}\\
&\hspace{3cm} - D(n_s,s)^{\frac{2}{2(t)}} \|u\|_{2t}^{\frac{4t}{2(t)}} \hspace{5.8cm}\text{\rm (by \eqref{eq:conditiononu} and \eqref{eq:canbang})}\\
&=D(n_s,s)^{\frac 2{2(t)}} \hatde[u],
\end{align*}
the last equality follow from the equality
\[
A_{n,t}^{\frac{4t}{2(t)}} = S_{n,s}(n_s-2)^2 D(n_s,s)^{1 -\frac{2}{2(t)}} \frac{n-nt+4t}{2(t+1)} \lt(\frac{t^2-1}{2n}\rt)^{\frac{2\theta t}{2(t)}},
\]
which can be checked by using $v$ as a test function (for which we have equality in \eqref{eq:DD}).
\end{proof} 
Proposition \eqref{bridge} combined with stability version of the weighted Sobolev inequality \eqref{eq:Seuffert} asserts the existence of a positive constant $C$ depending only on $n$ and $t$ such that
\begin{equation}\label{eq:startstep}
C \hatde[u]  \geq \inf_{c\in \R, a>0, x_0\in \R^n} \|\na f - \na g_{c,a,x_0}\|_2^2,
\end{equation}
with $u$ satisfies \eqref{eq:conditiononu}. Note that the normalized condition $\|u\|_{2t} = \|v\|_{2t}$ is equivalent to $\|f\|_{2_s^*} = \|g_{1,a,x_0}\|_{2_s^*}$ for any $a>0$ and $x_0 \in \R^n$. Our main goal of this section is to show that , up to enlarging the constant $C$, we can assume that $c = a =1$ in \eqref{eq:startstep}. This paves the way for the estimation on the infimum on the right hand side of \eqref{eq:startstep} in terms of $u$ and $v$. This point is different with the approach of Carlen and Figalli \cite{Carlen13} and of Seuffert \cite{Seuffert2}. In fact, after using the Bianchi--Egnell type stability version for the weighted Sobolev inequality, these authors continued using the weighted Sobolev inequality to estimate the deficit $\hatde[u]$ from below by the quantity
\[
\inf_{c\in \R, a >0, x_0\in \R^n} \|f - g_{c,a,x_0}\|_{2_s^*} ^{\frac{2}{2_s^*}}.
\]
and then applied their results to derive the stability version of GNS inequality.

\subsection{Controlling the infimum in the stability estimate of Bianchi--Egnell type}
The main result of this section reads as follows.

\begin{lemma}\label{enlarging}
Let $u\in \mathcal D^{2,t+1}(\R^n)$ be nonnegative function satisfying \eqref{eq:conditiononu}. Let $f$ define by \eqref{eq:defineoff}. Then there exists a constant $C_0, \de_0$ depending only on $n$ and $t$ such that for any real number $\delta >0$ with $\delta \leq \delta_0$ and 
\[
\|\nabla f -\nabla g_{c,a,x_0}\|_2^2 \leq \delta,
\]
for some $c\in \R, a >0,$ and $x_0\in \R^n$, then
\[
\|\na f - \na g_{1,1,x_0}\|_2^2 \leq C_0 \delta.
\]
\end{lemma}
Note that the values of $C_0$ and $\de_0$ can be computed explicitly from the proof below.
%In the sequel we use $C, C_0, K,$ to denote a positive constant which depends only on $n$ and $t$ and whose value can exchange from lines to lines (and even in the same line).
\begin{proof}
We follow the argument in the proof of Lemma $2.3$ in \cite{Carlen13}. Suppose that 
\[
\|\na f -\na g_{c,a,x_0}\|_2^2 \leq \delta
\]
for some $c \in \R, a >0$ and $x_0\in \R^n$. A simple compuation shows that
\[
\na f(x,y) = -\frac{n_s-2}2\lt((1-t)u(x)^{-t} \na u(x), 2y\rt) (u(x)^{1-t}+y^2)^{-\frac{n_s}2},
\]
and
\[
\na g_{c,a,x_0}(x,y) = -(n_s-2) c\, a^{\frac{n_s+2}2} \lt(x-x_0,y\rt) \lt(1+a^2|x-x_0|^2 + a^2 y^2\rt)^{-\frac{n_s}2}.
\]
We divide our proof into several steps.\\

$\bullet$ \emph{Step $1$: There exists $\de_1$ depending on $n$ and $t$ such that whenever $\de \leq \de_1$ we have $c > 0$.\text} Indeed, if $c\leq 0$, then
\begin{align*}
\de &\geq (n_s-2)^2\int_{\R^{n+1}_+} \lt|\frac{1}{(u(x)^{1-t}+y^2)^{n_s/2}} -\frac{c\, a^{(n_s+2)/2}}{(1+a^2|x-x_0|^2 +a^2 y^2)^{n_s/2}}\rt|^2 y^{2+s} dx dy\\
&\geq (n_s-2)^2\int_{\R^{n+1}_+} (u(x)^{1-t}+y^2)^{-n_s} y^{2+s} dx dy\\
&= (n_s-2)^2 D(n_s,2+s) \|u\|_{t+1}^{t+1}\\
&= (n_s-2)^2 D(n_s,2+s) \lt(\frac{t^2 -1}{2n}\rt)^{\frac{2 \theta t}{2(t)}} \|\na u\|_2^{\frac{4\theta t}{2(t)}} \|u\|_{t+1}^{\frac{4(1-\theta) t}{2(t)}} \qquad\qquad \text{\rm (by \eqref{eq:conditiononu} and \eqref{eq:canbang}}\\
&\geq (n_s-2)^2 D(n_s,2+s) \lt(\frac{t^2 -1}{2n}\rt)^{\frac{2 \theta t}{2(t)}} \lt(\frac{\|u\|_{2t}}{A_{n,t}}\rt)^{\frac{4t}{2(t)}}\hspace{2.2cm}\text{\rm (by GNS inequality)}\\
&= (n_s-2)^2 D(n_s,2+s) \lt(\frac{t^2 -1}{2n}\rt)^{\frac{2 \theta t}{2(t)}} \lt(\frac{\|v\|_{2t}}{A_{n,t}}\rt)^{\frac{4t}{2(t)}}\hspace{4.5cm} \text{\rm (by \eqref{eq:conditiononu}}.
\end{align*}
Hence, if 
\[
\de \leq \de_1 := (n_s-2)^2 D(n_s,2+s) \lt(\frac{t^2 -1}{2n}\rt)^{\frac{2 \theta t}{2(t)}} \lt(\frac{\|v\|_{2t}}{A_{n,t}}\rt)^{\frac{4t}{2(t)}},
\]
we then obtain a contradiction.\\

$\bullet$ \emph{Step $2$: Assume that $\de\leq \de_1$, then we have $\|\na f -\na g_{1,a,x_0}\|_2^2 \leq 4\de$.} Indeed, by \emph{Step $1$} we have $c > 0$. Note that $\|\na g_{c,a,x_0}\|_2= c\|\na g_{1,a,x_0}\|_2 = c \|\na g_{1,1,0}\|_2$ for any $c >0$. Hence 
\[
|c-1|\|\na g_{1,1,0}\|_2 = \lt|\|\na g_{c,a,x_0}\|_2 -\|\na f\|_2\rt|\leq \|\na g_{c,a,x_0} -\na f\|_2 \leq \de^{\frac12},
\]
and by triangle inequality, we get
\begin{align*}
\|\na f -\na g_{1,a,x_0}\|_2 &\leq \|\na f -\na g_{c,a,x_0}\|_2 + \|\na g_{c,a,x_0} -\na g_{1,a,x_0}\|_2\\
&\leq \de^{\frac12} + |c-1|\|\na g_{1,a,x_0}\|_2\\
&\leq 2 \de^{\frac12},
\end{align*}
which implies our desired estimate.

$\bullet$ \emph{Step $3$: There exist $\de_2$ and $C_2$ depending only on $n$ and $t$ such that whenever $\de \leq \de_2$, we then have
\[
\|\na f - \na g_{1,1,x_0}\|_2^2 \leq C_2 \de.
\]
\text}
By \emph{Step $2$}, we have
\[
\|\na f - \na g_{1,a,x_0}\|_2^2 \leq 4 \de,
\]
if $\de \leq \de_1$. Hence
\begin{align*}
4\de &\geq (n_s-2)^2 \int_{\R^{n+1}_+} \lt|\frac{1}{(u(x)^{1-t}+y^2)^{n_s/2}} -\frac{a^{(n_s+2)/2}}{(1+a^2|x-x_0|^2 +a^2 y^2)^{n_s/2}}\rt|^2 y^{2+s} dx dy\\
&= (n_s-2)^2 \int_{\R^{n+1}_+} \lt|\frac{1}{a^{(n_s+2)/2}(u(x/a)^{1-t}+y^2/a^2)^{n_s/2}} -\frac{1}{(1+|x-a \,x_0|^2 +y^2)^{n_s/2}}\rt|^2 y^{2+s} dx dy.
\end{align*}
Let 
\[
A =\{(x,y)\in \R^{n+1}_+\,:\, |x-a x_0|\leq 1,\, 0< y\leq 1\}.
\]
Note that
\begin{equation}\label{eq:volumeA}
\int_A y^{2+s} dx dy = \frac{1}{s+3}\, \om_n,\qquad \omega_n = \frac{\pi^{n/2}}{\Gamma(1+n/2)}.
\end{equation}
By Fubini's theorem, for any set $B \subset A$ with 
\[
\int_B y^{s+2} dy dx > \frac1{s+3} \lt(1 -\frac1{4^{s+3}}\rt) \, \om_n
\]
there exists $\o{x}$ such that $|\o x-a\, x_0|\leq 1$ such that there exist $0 < y_1 <1/4$ and $3/4 < y_2 <1$ such that $(\o x,y_1), (\o x, y_2) \in B$. Indeed, if this is not the case, then for any a.e $x$ such that $|x -a\, x_0|\leq 1$, the set $B\cap x \times (0,1)$ has empty intersection with at least one of $x \times (0,1/4)$ or $x\times (3/4,1)$,
\begin{align*}
\int_B y^{s+2} dy ds& = \int_{\{|x-ax_0|\leq 1\}} \int_{B\cap x\times (0,1)} y^{s+2} dy dx\\
&\leq \frac1{s+3} \lt(1 - \frac1{4^{s+3}}\rt)\, \om_n,
\end{align*}
with contradicts with our assumption on $B$. Evidently, we have
\[
\frac{4\de}{(n_s-2)^2} \geq \int_{A} \lt|\frac{1}{a^{(n_s+2)/2}(u(x/a)^{1-t}+y^2/a^2)^{n_s/2}} -\frac{1}{(1+|x-a \,x_0|^2 +y^2)^{n_s/2}}\rt|^2 y^{2+s} dx dy.
\]
For any fixed $\gamma$, denote
\[
C =\lt\{(x,y)\in A\, :\, \lt|\frac{1}{a^{(n_s+2)/2}(u(x/a)^{1-t}+y^2/a^2)^{n_s/2}} -\frac{1}{(1+|x-a \,x_0|^2 +y^2)^{n_s/2}}\rt|^2 \geq \gamma \de\rt\}.
\]
Applying Chebyshev's inequality, we have
\[
\int_C y^{s+2}dy dx \leq \frac{4}{\gamma (n_s-2)^2}.
\]
Choosing
\[
\gamma = \frac{2(s+3) 4^{s+4}}{(n_s-2)^2 \om_n},\qquad B = A\setminus C.
\]
Then we have
\[
\int_B y^{s+2}dy dx > \frac1{s+3} \lt(1 -\frac1{4^{s+3}}\rt) \, \om_n,
\]
and for any $(x,y) \in B$,
\begin{equation}\label{eq:cru1}
\lt|\frac{1}{a^{(n_s+2)/2}(u(x/a)^{1-t}+y^2/a^2)^{n_s/2}} -\frac{1}{(1+|x-a \,x_0|^2 +y^2)^{n_s/2}}\rt| \leq \gamma^{\frac12} \de^{\frac12}.
\end{equation}
Notice that for $(x,y) \in A$, we have $1+|x-a \,x_0|^2 +y^2 \leq 3$. Hence, if 
\[
\de \leq \de_1': = \frac1{4\times 3^{n_s} \gamma}, 
\]
we get from \eqref{eq:cru1} for any $(x,y) \in B$ that
\begin{align*}
\gamma^{\frac12} \de^{\frac12} &\geq \frac{1}{(1+|x-a \,x_0|^2 +y^2)^{n_s/2}} - \frac{1}{a^{(n_s+2)/2}(u(x/a)^{1-t}+y^2/a^2)^{n_s/2}}\\
&\geq 3^{-\frac{n_s}2} - \frac{1}{a^{(n_s+2)/2}(u(x/a)^{1-t}+y^2/a^2)^{n_s/2}},
\end{align*}
and
\begin{align*}
\frac{1}{a^{(n_s+2)/2}(u(x/a)^{1-t}+y^2/a^2)^{n_s/2}}& \leq \frac{1}{(1+|x-a \,x_0|^2 +y^2)^{n_s/2}} + \gamma^{\frac12}\de^{\frac12}\\
&\leq 1+ \frac1{2\times 3^{n_s/2}}\\
& < 2.
\end{align*}
Thus, we obtain
\[
\frac12 \leq a^{\frac{n_s+2}2}\lt(u\lt(\frac x a\rt)^{1-t}+\frac{y^2}{a^2}\rt)^{\frac{n_s}2} \leq 2\times 3^{\frac{n_s}2},
\]
and hence
\begin{equation}\label{eq:cru2}
\lt|a^{\frac{n_s+2}2}(u(x/a)^{1-t}+y^2/a^2)^{\frac{n_s}2} -(1+|x-a \,x_0|^2 +y^2)^{\frac{n_s}2} \rt| \leq 2\times 3^{n_s} \gamma^{\frac12} \delta^{\frac12},
\end{equation}
for any $(x,y) \in B$. It is an elementary estimation that
\begin{align}\label{eq:cru3}
\Bigg|a^{\frac{n_s+2}2}\Bigl(u\Big(\frac xa\Big)^{1-t}&+\frac{y^2}{a^2}\Bigl)^{\frac{n_s}2} -(1+|x-a \,x_0|^2 +y^2)^{\frac{n_s}2} \Bigg| \notag\\
&\geq \frac{n_s}2 \Bigg|a^{1+\frac2{n_s}}(u(x/a)^{1-t}+y^2/a^2) -(1+|x-a \,x_0|^2 +y^2)\Bigg| \notag\\
&\qquad \times \min\{a^{1+\frac2{n_s}}(u(x/a)^{1-t}+y^2/a^2),1+|x-a \,x_0|^2 +y^2\}^{\frac{n_s}2 -1}\notag\\
&\geq \frac{n_s}2 2^{\frac2{n_s}-1} \lt|a^{1+\frac2{n_s}}\lt(u\lt(\frac xa\rt)^{1-t}+\frac{y^2}{a^2}\rt) -(1+|x-a \,x_0|^2 +y^2)\rt|,
\end{align}
for $(x,y)\in B$. Combining \eqref{eq:cru2} and \eqref{eq:cru3}, we get
\begin{equation}\label{eq:cru4}
\lt|a^{1+\frac2{n_s}}\lt(u\lt(\frac xa\rt)^{1-t}+\frac{y^2}{a^2}\rt) -(1+|x-a \,x_0|^2 +y^2)\rt| \leq \frac{2^3 3^{n_s}}{n_s}\gamma^{\frac12} \de^{1/2},
\end{equation}
for $(x,y)\in B$. 

Our observation above shows that we can choose $\o x$ with $|\o x -a\, x_0|\leq 1$ and  $0< y_1 < 1/4$ and $ 3/4 < y_2 < 1$ such that $(\o x,y_1), (\o x ,y_2) \in B$. Then by \eqref{eq:cru4}, we have
\begin{align*}
\frac12 \lt|\frac{1}{a^{1-2/n_s}} -1\rt| &\leq |y_1^2 -y_2^2|\, \lt| \frac{1}{a^{1-2/n_s}} -1\rt|\\
&=\Bigg|\Bigg[a^{1+\frac2{n_s}}\Bigl(u\Big(\frac {\o x}a\Big)^{1-t}+\frac{y_1^2}{a^2}\Bigl) -(1+|\o x-a \,x_0|^2 +y_1^2)\Bigg] \\
&\qquad - \Bigg[a^{1+\frac2{n_s}}\Bigl(u\Big(\frac {\o x}a\Big)^{1-t}+\frac{y_2^2}{a^2}\Bigl) -(1+|\o x-a \,x_0|^2 +y_2^2)\Bigg]\Bigg|\\
&\leq \lt|a^{1+\frac2{n_s}}\Bigl(u\Big(\frac {\o x}a\Big)^{1-t}+\frac{y_1^2}{a^2}\Bigl) -(1+|\o x-a \,x_0|^2 +y_1^2)\rt| \\
&\qquad + \lt|a^{1+\frac2{n_s}}\Bigl(u\Big(\frac {\o x}a\Big)^{1-t}+\frac{y_2^2}{a^2}\Bigl) -(1+|\o x-a \,x_0|^2 +y_2^2)\rt|\\
&\leq \frac{2^4 3^{n_s}}{n_s}\gamma^{\frac12} \de^{1/2}.
\end{align*}
Thus, we get
\begin{equation}\label{eq:cru5}
\lt|\frac{1}{a^{1-2/n_s}} -1\rt| \leq \frac{2^5 3^{n_s}}{n_s}\gamma^{\frac12} \de^{1/2} =:\gamma_1^{\frac12} \de^{\frac12}.
\end{equation}
Hence, if
\[
\de \leq \delta_1'' : = \frac{n_s^2}{2^{12} 3^{2n_s} \gamma }
\]
we then have $1/2 \leq a^{-1+ 2/n_s} \leq 2$, or equivalently
\begin{equation}\label{eq:bounda}
2^{-\frac{n_s}{n_s-2}} \leq a \leq 2^{\frac{n_s}{n_s -2}}.
\end{equation}
%Combining \eqref{eq:cru5} and \eqref{eq:bounda}, we get
%\begin{equation}\label{eq:agan1}
%\end{equation}
%Since 
%\[
%\lt|\frac{1}{a^{1-2/n_s}} -1\rt| \geq \frac{n_s-2}{n_s} |a -1| \lt(\max\{ a, 1\}\rt)^{\frac2{n_s} -2} \geq \frac{n_s-2}{n_s} 2^{-\frac{2(n_s-1)}{n_s-2}} |a-1|.
%\]
%This estimate and \eqref{eq:cru5} implies
%\[
%|a-1| \leq \frac{n_s}{n_s-2} 2^{\frac{2(n_s-1)}{n_s-2}}\frac{2^5 3^{n_s}}{n_s}\gamma^{\frac12} \de^{1/2} =: \gamma_1^{\frac12} \de^{\frac12}
%\]
%if $\de \leq \min\{\de_1, \de_2, \de_3\}$.
We continue our proof by bounding $\|\na g_{1,1,x_0} -\na g_{1,a,x_0}\|_2^2$. By translating in variable $x$, it is enough to bound $\|\na g_{1,1,0} -\na g_{1,a,0}\|_2^2$. A straightforward computation shows that
\begin{multline*}
\|\na g_{1,1,0} -\na g_{1,a,0}\|_2^2\\
=(n_s-2)^2 \int_{\R^{n+1}_+} (|x|^2 +y^2)\Bigg|\frac1{(1+|x|^2 +y^2)^{n_s/2}} -\frac{a^{(n_s+2)/2}}{(1+a^2|x|^2 +a^2 y^2)^{n_s/2}}\Bigg|^2 y^s dx dy.
\end{multline*}
Notice that
\begin{align*}
&\Bigg|\frac1{(1+|x|^2 +y^2)^{n_s/2}}-\frac{a^{(n_s+2)/2}}{(1+a^2|x|^2 +a^2 y^2)^{n_s/2}}\Bigg| \\
&\leq \frac{n_s}2\Bigg|\frac1{1+|x|^2 +y^2}-\frac{a^{1+2/n_s}}{1+a^2|x|^2 +a^2 y^2}\Bigg| \max\lt\{\frac1{1+|x|^2 +y^2}, \frac{a^{1+2/n_s}}{1+a^2|x|^2 +a^2 y^2}\rt\}^{\frac{n_s}2-1}\\
&\leq n_s 2^{\frac{n_s}2-2} \Bigg|\frac1{1+|x|^2 +y^2}-\frac{a^{1+2/n_s}}{1+a^2|x|^2 +a^2 y^2}\Bigg| \lt(\frac1{1+|x|^2 +y^2}\rt)^{\frac{n_s}2 -1}\hspace{1cm}\text{\rm (by \eqref{eq:bounda})}\\
&= n_s 2^{\frac{n_s}2-2} \frac{a^2|1-a^{-1+2/n_s}|(|x|^2+y^2)}{(1+|x|^2 +y^2)(1+a^2|x|^2 +a^2 y^2)}\lt(\frac1{1+|x|^2 +y^2}\rt)^{\frac{n_s}2 -1}\\
&\leq n_s 2^{\frac{n_s}2-2} 2^{\frac{4n_s}{n_s-2}} |1-a^{-1+2/n_s}|(|x|^2+y^2)\lt(\frac1{1+|x|^2 +y^2}\rt)^{\frac{n_s}2 +1} \hspace{1.9cm}\text{\rm (by \eqref{eq:bounda})}\\
&\leq n_s 2^{\frac{n_s}2-2} 2^{\frac{4n_s}{n_s-2}} |1-a^{-1+2/n_s}|\lt(\frac1{1+|x|^2 +y^2}\rt)^{\frac{n_s}2}
\end{align*}
Hence
\begin{align*}
\|\na g_{1,1,0} -\na g_{1,a,0}\|_2^2&\leq (n_s-2)^2n_s^2 2^{n_s-4} 2^{\frac{bn_s}{n_s-2}} |1-a^{-1+2/n_s}|^2\int_{\R^{n+1}_+}\lt(\frac1{1+|x|^2 +y^2}\rt)^{n_s-1} y^s dy dx\\
&\leq (n_s-2)^2n_s^2 2^{n_s-4} 2^{\frac{bn_s}{n_s-2}} \frac{\pi^{n/2}\Gamma((1+s)/2) \Gamma((n_s-2)/2)}{2\Gamma(n_s-1)}\gamma_1 \de\\
&=: \gamma_2 \de,
\end{align*}
here, the second inequality comes from \eqref{eq:cru5} and computing explicitly the integral. Finally, by triangle inequality and Cauchy--Schwartz inequality, we have
\[
\|\na f -\na g_{1,1,x_0}\|_2^2 \leq 2(\|\na f-\na g_{1,a,x_0}\|_2^2 + \|\na g_{1,a,x_0} -\na g_{1,1,x_0}\|_2^2) \leq 2(4+ \gamma_2)\de.
\]
This finishes the proof of \emph{Step $3$} by choosing $\de_2 =\min\{\de_1, \de_1',\de_1''\}$ and $C_2 =2(4+\gamma_2)$. 
\end{proof}

\subsection{Proof of Theorem \ref{1stTheorem}}
%Notice first that $4t > 2(t)$ then by the convexity, we have
%\[
%\hatde[u] \geq \frac{4t}{2(t)}\|u\|_{2t}^{\frac{4t}{2(t)}-1}\deGNS[u]  
%\]
%and
%\[
%\hatde[u] \leq \frac{4t}{2(t)} \lt(A_{n,t} \|\na u\|_2^\theta \|u\|_{t+1}^{1-\theta}\rt)^{\frac{4t}{2(t)} -1} \deGNS[u] =\frac{4t}{2(t)} \lt(\|u\|_{2t} + \deGNS[u]\rt)^{\frac{4t}{2(t)} -1} \deGNS[u].
%\]
%Thus, if $\|u\|_{2t} = \|v\|_{2t}$ and $\deGNS[u] \leq 1$, there exists constant $B_{n,t} > 1$ depending only on $n$ and $t$ such that
%\begin{equation}\label{eq:sosanhdeficit}
%\frac1{B_{n,t}} \deGNS[u] \leq \hatde[u] \leq B_{n,t} \deGNS[u].
%\end{equation}
Let $C_0$ and $\de_0$ be as in Lemma \ref{enlarging}. Define
\[
\de = \min\lt\{1, \frac{\de_0}{C}\rt\},
\]
with $C$ appears in \eqref{eq:startstep}. Suppose that $u$ satisfies the condition of Theorem \ref{1stTheorem}, and $\hatde[u] \leq \de$, hence $C\hatde[u] \leq \de_0$. By \eqref{eq:startstep}, there exist $c,a,x_0$ such that
\[
\|\na f -\na g_{c,a,x_0}\|_2^2 \leq C\hatde[u] \leq \de_0.
\]
Applying Lemma \ref{enlarging}, we get
\[
\|\na f - \na g_{1,1,x_0}\|_2^2 \leq C_0 C \hatde[u]=: K_1 \hatde[u],
\]
Replacing $u$ by $u(\cdot + x_0)$ which does not change $\deGNS[u]$, we can asumme that $x_0 =0$ (we make this assumption for simplifying the notation in the proof). Our aim is to prove the following inequality
\begin{equation}\label{eq:aim}
\int_{\R^n} |\na u(x) -\na v(x)|^2 dx + \int_{\R^n} \lt|u(x)^{\frac{1+t}2} -v(x)^{\frac{1+t}2}\rt|^2 dx \leq K \hatde[u],
\end{equation}
for some constant $K$ depending only on $n$ and $t$. Notice that
\[
g_{1,1,0} (x,y) = (v(x)^{1-t} + y^2)^{-\frac{n_s-2}2},
\]
and then by an easy computation, we have
\[
\na g_{1,1,0}(x,y) = -\frac{n_s-2}2 \lt((1-t)v(x)^{-t} \na v(x), 2y\rt)(v(x)^{1-t} + y^2)^{-\frac{n_s}2}.
\]
Substituting the expressions of $\na f$ and $\na g_{1,1,0}$ into $\|\na f-\na g_{1,1,0}\|_2$, we have
\begin{align*}
\|\na f - \na g_{1,1,0}\|_2^2 &= (1-t)^2\frac{(n_s-2)^2}4 \int_{\R^{n+1}_+}\lt|\frac{u(x)^{-t}\na u(x)}{(u(x)^{1-t} + y^2)^{\frac{n_s}2}} -\frac{v(x)^{-t}\na v(x)}{(v(x)^{1-t} + y^2)^{\frac{n_s}2}}\rt|^2 y^s dx dy\\
&\qquad + (n_s-2)^2 \int_{\R^{n+1}_+}\lt|\frac1{(u(x)^{1-t} + y^2)^{\frac{n_s}2}} -\frac1{(v(x)^{1-t} + y^2)^{\frac{n_s}2}}\rt|^2 y^{2+s} dx dy\\
&=(1-t)^2\frac{(n_s-2)^2}4 \, I + (n_s-2)^2 \, II.
\end{align*}
Our goal is to estimate $I$ and $II$.

\emph{$\bullet$ Estimate $II$:} By a simple change of variable, we have
\[
\int_{0}^\infty \frac1{(a^{1-t}+ y^2)^{n_s}} y^s dx dy = a^{1+t} \int_{0}^\infty \frac1{(1+ y^2)^{n_s}} y^s dx dy =: \alpha_{n,s} a^{1+t},\qquad \forall\, a>0.
\]
Hence
\begin{align*}
II &= \alpha_{n,s} \lt(\int_{\R^n} u(x)^{1+t} dx + \int_{\R^n} v(x)^{1+t}dx\rt) \\
&\qquad\qquad -2\int_{\R^{n+1}_+} \lt(\frac1{(u(x)^{1-t} + y^2)(v(x)^{1-t} +y^2)}\rt)^{\frac{n_s}2} y^s dy dx\\
&=\alpha_{n,s} \lt(\int_{\R^n} u(x)^{1+t} dx + \int_{\R^n} v(x)^{1+t}dx\rt)\\
&\qquad\qquad -2\int_{\R^{n+1}_+} \lt(\frac1{(u(x)^{1-t}v(x)^{1-t} + y^2(u(x)^{1-t}+v(x)^{1-t}) +y^4)}\rt)^{\frac{n_s}2} y^s dy dx\\
&\geq \alpha_{n,s} \lt(\int_{\R^n} u(x)^{1+t} dx + \int_{\R^n} v(x)^{1+t}dx\rt)\\
&\qquad\qquad -2\int_{\R^{n+1}_+}\frac{1}{(u(x)^{(1-t)/2} v(x)^{(1-t)/2} + y^2)^{n_s}} y^s dy dx\\
&= \alpha_{n,s} \int_{\R^n} \lt|u(x)^{\frac{1+t}2} -v(x)^{\frac{1+t}2}\rt|^2 dx.
\end{align*}
Thus, we have shown that
\begin{equation}\label{eq:norm1+t}
\int_{\R^n} \lt|u(x)^{\frac{1+t}2} -v(x)^{\frac{1+t}2}\rt|^2 dx \leq \frac{C_1}{\alpha_{n,s} (n_s-2)^2} \, \hatde[u] = : K_2 \hatde[u].
\end{equation}

\emph{$\bullet$ Estimate $I$:} By triangle inequality, we have
\begin{align*}
I^{\frac12} &\geq \lt(\int_{\R^{n+1}_+}\lt|\frac{u(x)^{-t}(\na u(x) -\na v(x))}{(u(x)^{1-t} + y^2)^{\frac{n_s}2}} \rt|^2 y^s dx dy\rt)^{\frac12} \\
&\qquad -\lt(\int_{\R^{n+1}_+}|\na v(x)|^2\lt|\frac{u(x)^{-t}}{(u(x)^{1-t} + y^2)^{\frac{n_s}2}} -\frac{v(x)^{-t}}{(v(x)^{1-t} + y^2)^{\frac{n_s}2}}\rt|^2 y^s dx dy\rt)^{\frac12}\\
&= \lt(\beta_{n,s} \int_{\R^n} |\na u(x) -\na v(x)|^2 dx\rt)^{\frac12} -III^{\frac12},
\end{align*}
with 
\[
\beta_{n,s} = \int_{0}^\infty \frac1{(1+ y^2)^{n_s}} y^s dy.
\]
By a suitable change of variable, we have
\begin{align*}
III&= 2\int_{\R^n} |\na v(x)|^2 u(x)^{-t} v(x)^{-t}\\
&\qquad \qquad\times \int_0^{\infty} \Bigg(\frac1{((u(x) v(x))^{\frac{1-t}2} + y^2)^{n_s}}-\frac1{(v(x)^{1-t} + y^2)^{\frac{n_s}2}(u(x)^{1-t} + y^2)^{\frac{n_s}2}}\Bigg) y^s dx dy\\
&=2\int_{\R^n} |\na v(x)|^2 u(x)^{\frac{(s+1)(t-1)}4} v(x)^{\frac{(s+1)(t-1)}4}\\
&\qquad\times \int_0^{\infty} \Bigg(\frac1{(1+ (u(x) v(x))^{\frac{t-1}2}y^2)^{n_s}}-\frac1{(1+ v(x)^{t-1} y^2)^{\frac{n_s}2}(1+ u(x)^{t-1}y^2)^{\frac{n_s}2}}\Bigg) y^s dx dy.
\end{align*}
Notice that $(1+ (u(x) v(x))^{\frac{t-1}2}y^2)^2 \leq (1+ u(x)^{t-1}y^2)(1+ v(x)^{t-1} y^2)$, by the convexity we have
\begin{align*}
&\frac1{(1+ (u(x) v(x))^{\frac{t-1}2}y^2)^{n_s}} -\frac1{(1+ v(x)^{t-1} y^2)^{\frac{n_s}2}(1+ u(x)^{t-1}y^2)^{\frac{n_s}2}}\\
&\leq \frac{n_s}2 \lt[\frac1{(1+ (u(x) v(x))^{\frac{t-1}2}y^2)^2} -\frac1{(1+ u(x)^{t-1}y^2)(1+ v(x)^{t-1} y^2)}\rt]\frac1{(1+ (u(x) v(x))^{\frac{t-1}2}y^2)^{n_s-2}}\\
&= \frac{n_s}2 \frac{|u(x)^{\frac{t-1}2} -v(x)^{\frac{t-1}2}|^2 y^2}{(1+ (u(x) v(x))^{\frac{t-1}2}y^2)^{n_s}[1+ (u(x)^{t-1} + v(x)^{t-1})y^2 + u(x)^{t-1}v(x)^{t-1} y^4]}\\
&\leq \frac{n_s}2 \frac{|u(x)^{\frac{t-1}2} -v(x)^{\frac{t-1}2}|^2}{u(x)^{t-1} + v(x)^{t-1}} \frac1{(1+ (u(x) v(x))^{\frac{t-1}2}y^2)^{n_s}}.
\end{align*}
Thus we get
\begin{align*}
III&\leq n_s\int_{\R^n} |\na v(x)|^2 \frac{|u(x)^{\frac{t-1}2} -v(x)^{\frac{t-1}2}|^2}{u(x)^{t-1} + v(x)^{t-1}} \int_0^\infty \frac{u(x)^{\frac{(s+1)(t-1)}4} v(x)^{\frac{(s+1)(t-1)}4}}{(1+ (u(x) v(x))^{\frac{t-1}2}y^2)^{n_s}} y^s dy dx\\
&= n_s \beta_{n,s} \int_{\R^n}|\na v(x)|^2 \frac{|u(x)^{\frac{t-1}2} -v(x)^{\frac{t-1}2}|^2}{u(x)^{t-1} + v(x)^{t-1}} dx.
\end{align*}
Since $v(x) = (1+ |x|^2)^{1/(1-t)}$, then $\na v(x) = [2x (1+ |x|^2)^{t/(1-t)}]/(1-t)$, and
\[
|\na v(x)| = \frac2{t-1} |x| (1+ |x|^2)^{-\frac{t}{t-1}} \leq \frac{2}{t-1} v(x)^{\frac{t+1}{2}}.
\]
Hence, there exists $\beta'_{n,s}$ depending only on $n$ and $t$ such that
\[
III \leq \be'_{n,s} \int_{\R^n} |u(x)^{\frac{t+1}2} -v(x)^{\frac{t+1}2}|^2 dx \leq \beta'_{n,s} K_2 \hatde[u] =: K_3\hatde[u],
\]
here we use \eqref{eq:norm1+t}. Applying Cauchy-Schwartz inequality, we have
\begin{align}\label{eq:L2gradient}
\int_{\R^n} |\na u(x) -\na v(x)|^2 dx &\leq \frac{2}{\beta_{n,s}}(I + III)\notag\\
& \leq \frac{2}{\beta_{n,s}}\lt(\frac{4}{(t-1)^2 (n_s-2)^2} K_1+ K_3\rt) \hatde[u][u] \notag\\
&=: K_4 \hatde[u].
\end{align}
Let $K = K_4 + K_2$ which depends only on $n$ and $t$, we have
\[
\int_{\R^n} |\na u(x) -\na v(x)|^2 dx + \int_{\R^n} \lt|u(x)^{\frac{1+t}2} -v(x)^{\frac{1+t}2}\rt|^2 dx \leq K \hatde[u]
\]
as our desire \eqref{eq:aim}. This finishes our proof of Theorem \ref{1stTheorem}.

\begin{remark}\label{remark}
Since for any $x_0 \in R^n$
\[
\|\na u-\na v_{1,x_0}\|_2^2 \leq 2(\|\na u\|_2^2 + \|\na v\|_2^2)\quad\text{and}\quad \|u^{\frac{1+t}2} -v_{1,x_0}^{\frac{1+t}2}\|_2^2 \leq 2(\|u\|_{t+1}^{t+1} + \|v\|_{t+1}^{t+1}),
\]
hence
\begin{align*}
\inf_{x_0\in \R^n}\{\|\na u-\na v_{1,x_0}\|_2^2 + \|u^{\frac{1+t}2} -v_{1,x_0}^{\frac{1+t}2}\|_2^2\} \leq A + B\, \|u\|_{t+1}^{t+1}\leq C(1 + \hatde[u]),
\end{align*}
with $A,B$ depend only on $n$ and $t$, here we use \eqref{eq:conditiononu}. So, up to enlarging the constant $K$, Theorem \ref{1stTheorem} always holds without restriction on $\hatde[u]$.
\end{remark}

We conclude this section by showing how our Theorem \ref{1stTheorem} implies the stability result of Carlen and Figalli, and of Seuffert \eqref{eq:Seuffert}. Indeed, suppose that $\|u\|_{2t} = \|v\|_{2t}$. For $a >0$, define $u_a(x) = a^{n/2t} u(ax)$, then $\hatde[u_a] = \hatde[u]$. By a suitable choice of $a >0$, we have $[(t^2 -1)\|\na u_a\|_2^2]/(2n) = \|u_a\|_{t+1}^{t+1}$. Applying Theorem \ref{1stTheorem}, then whenever 
\[
\hatde[u_a]=\hatde[u] \leq \de,
\]
we can choose $x_0\in \R^n$ such that
\begin{equation}\label{eq:deriveCFS}
\int_{\R^n} |\na u_a(x) -\na v_{1,x_0}(x)|^2 dx + \int_{\R^n} \lt|u_a(x)^{\frac{1+t}2} -v_{1,x_0}(x)^{\frac{1+t}2}\rt|^2 dx \leq K \hatde[u].
\end{equation}
Since $\|u_a\|_{2t} = \|v\|_{2t}$, then 
\[
A_{n,t}^{\frac{4t}{2(t)}} \|\na u_a\|_2^{\frac{4 \theta t}{2(t)}} \|u_a\|_{t+1}^{\frac{4(1-\theta)t}{2(t)}} = \|u\|_{2t}^{\frac{4t}{2(t)}} + \deGNS[u_a] \leq \|v\|_{2t}^{\frac{4t}{2(t)}} + \de.
\]
This and the fact $[(t^2 -1)\|\na u_a\|_2^2]/(2n) = \|u_a\|_{t+1}^{t+1}$ imply
\[
\|\na u_a\|_2 \leq C(n,t),\,  \quad\text{and}\quad\, \|u_a\|_{t+1} \leq C(n,t),
\]
for some constant $C(n,t)$ depending only on $n$ and $t$.\\

\emph{$\bullet$ If $n=2$:} By H\"older inequality and \eqref{eq:deriveCFS}, we have
\begin{align*}
\int_{\R^n} |u_a^{2t} -v_{1,x_0}^{2t}| dx &\leq c_1\int_{\R^n} \lt|u_a^{\frac{1+t}2} -v_{1,x_0}^{\frac{1+t}2}\rt|\, (u_a + v_{1,x_0})^{\frac{3t-1}2} dx\\
&\leq c_1\lt(\int_{\R^n} \lt|u_a^{\frac{1+t}2} -v_{1,x_0}^{\frac{1+t}2}\rt|^2 dx\rt)^{\frac12} \|u_a + v_{1,x_0}\|_{3t-1}^{\frac{3t-1}2}\\
&\leq c_2 \sqrt{K \hatde[u]}\lt(\|u_a\|_{3t-1}^{\frac{3t-1}2} + \|v_{1,x_0}\|_{3t-1}^{\frac{3t-1}2}\rt),
\end{align*}
with $c_1, c_2$ depend only on $t$. Since $3t -1 > t+1$, by GNS inequality, we have
\[
\|u_a\|_{3t-1} \leq C_{\rm GNS}(n,t) \|\na u_a\|_2^{\theta'} \|u_a\|_{t+1}^{1-\theta'}  \leq C(n,t) C_{\rm GNS}(n,t),
\]
with $\theta' = 2(t-1)/(3t-1)$. Thus there exists $K'$ depending only on $n,t$ such that
\[
\int_{\R^n} |u_a^{2t} -v_{1,x_0}^{2t}| dx \leq K' \hatde[u]^{\frac12},
\]
as our desire.\\

\emph{$\bullet$ If $n\geq 3$:} If $3t -1 \leq 2n/(n-2)$, or equivalently $t \leq (3n-2)/(3(n-2))$, then repeating the argument in the case $n =2$, we obtain \eqref{eq:Seuffert}. If $t > (3n-2)/(3(n-2))$, then we have by an easy computation that
\[
t+ 1 < (2t-1) \frac{2n}{n+2}=:q(t) \leq \frac{2n}{n-2}.
\]
From \eqref{eq:deriveCFS}, we have
\[
\lt(\int_{\R^n} |u_a -v_{1,x_0}|^{\frac{2n}{n-2}} dx\rt)^{\frac{n-2}n} \leq S_n\int_{\R^n} |\na u_a(x) -\na v_{1,x_0}(x)|^2 dx \leq S_n K \hatde[u],
\]
where $S_n$ is the best constant in Sobolev inequality. This and H\"older inequality implies
\begin{align*}
\int_{\R^n} |u_a^{2t} -v_{1,x_0}^{2t}| dx &\leq c_3\int_{\R^n} \lt|u_a -v_{1,x_0}\rt|\, (u_a + v_{1,x_0})^{2t-1} dx\\
&\leq c_3 \lt(\int_{\R^n} |u_a -v_{1,x_0}|^{\frac{2n}{n-2}} dx\rt)^{\frac{n-2}{2n}} \|u_a + v_{1,x_0}\|_{q(t)}^{2t-1}\\
&\leq c_4 \sqrt{S_n K \hatde[u]} \lt(\|u_a\|_{q(t)}^{2t-1} + \|v_{1,x_0}\|_{q(t)}^{2t-1}\rt),
\end{align*}
with $c_3, c_4$ depend only on $t$. Repeating the argument in the case $n=2$ by using GNS inequality, we obtain \eqref{eq:Seuffert}.

\subsection{Proof of Corollary \ref{quantitativeform}}
Observe that
\[
\hatde[u] = \|u\|_{2t}^{\frac{4t}{2(t)}} \lt( (1 + \deGNS[u])^{\frac{4t}{2(t)}} -1\rt).
\]
Since $4t > 2(t)$, then there exists a constant $B_{n,t}$ depending only on $n,t$ such that
\[
B_{n,t}^{-1} \deGNS[u] \leq \hatde[u] \leq B_{n,t}\deGNS[u],
\]
if $\deGNS[u] \leq 1$.

\begin{proof}[Proof of Corollary \ref{quantitativeform}]
We first suppose that $u$ is a nonnegative function with 
\[
\|u\|_{2t} = \|v\|_{2t},\quad \text{and}\quad \deGNS[u] \leq \de':=\min\lt\{1, \frac{\de}{B_{n,t}}\rt\}
\]
where $\de$ comes from Theorem \ref{1stTheorem}. Choose $a > 0$ such that 
\[
\frac{t^2-1}{2n} \|\na u_a\|_2^2 = \|u_a\|_{t+1}^{t+1},\quad\text{with}\quad u_a(x) =a^{\frac n{2t}} u(a x).
\]
Then $\deGNS[u_a] = \deGNS[u]$. The observation above implies $\hatde[u] \leq \de$. By Theorem \ref{1stTheorem}, there exists $x_0$ such that  
\begin{equation*}\label{eq:coroproof}
\int_{\R^n} |\na u_a(x) -\na v_{1,x_0}(x)|^2 dx + \int_{\R^n} \lt|u_a(x)^{\frac{1+t}2} -v_{1,x_0}(x)^{\frac{1+t}2}\rt|^2 dx \leq K \hatde[u].
\end{equation*}
In particular, we have $\|\na u_a -\na v_{1,x_0}\|_2 \leq \sqrt{K \hatde[u]}$ and 
\[
\|u_a -v_{1,x_0}\|_{t+1}^{t+1} \leq \|u_a(x)^{\frac{1+t}2} -v_{1,x_0}(x)^{\frac{1+t}2}\|_2^2 \leq K \hatde[u].
\]
By GNS inequality, we have
\[
\|u_a -v_{1,x_0}\|_{2t} \leq A_{n,t} \|\na u_a -\na v_{1,x_0}\|_2^\theta \|u_a -v_{1,x_0}\|_{t+1}^{1-\theta} \leq A_{n,t} (K \hatde[u])^{\frac\theta2+ \frac{1-\theta}{t+1}}.
\]
This implies our desired estimate \eqref{eq:quantiativeform} for nonnegative functions $u$ with $\|u\|_{2t} =\|v\|_{2t}$ and $\deGNS[u] \leq \de'$. Since for any function $u\in \mathcal D^{2,t+1}(\R^n)$ with $\|u\|_{2t} = \|v\|_{2t}$, we have
\[
\lambda_{\rm GNS}[u] \leq 2^{2t} \|u\|_{2t} =2^{2t} \|v\|_{2t}.
\]
Thus by enlarging the constant $C$, the inequality \eqref{eq:quantiativeform} holds for any nonnegative function $u$ such that $\|u\|_{2t} = \|v\|_{2t}$ without restriction on $\deGNS[u]$. 

The condition $\|u\|_{2t} = \|v\|_{2t}$ is removed by the homogeneity of \eqref{eq:quantiativeform}.

We next relax the assumption that $u$ is nonnegative. We follow the argument in \cite{Cianchi}. Denote
\[
p = \frac{2(t+1)}{\theta(t+1) + 2(1-\theta)},\qquad \alpha = \frac{(t+1)\theta}{\theta(t+1) + 2(1-\theta)} \in (0,1).
\]
We will show that there exists a constant $C$ depending only on $n, t$ such that
\begin{equation}\label{eq:onepartsmall}
\min\lt\{\frac1{\|u\|_{2t}^{2t}}\int_{\{0< u\}} |u|^{2t} dx ,\frac1{\|u\|_{2t}^{2t}}\int_{\{ u>0\}} |u|^{2t} dx\rt\} \leq C\deGNS[u]^{\frac{2t}p}.
\end{equation}
Applying GNS inequality for $u_+$ and $u_-$, we have
\[
\|u_*\|_{2t} \leq A_{n,t} (\|\na u_*\|_2^2)^{\frac{\theta}2} \, (\|u_*\|_{t+1}^{t+1})^{\frac{1-\theta}{t+1}},
\]
with $* = \pm$. This implies
\[
\|u_*\|_{2t}^p \leq A_{n,t}^p (\|\na u_*\|_2^2)^{\alpha} (\|u_*\|_{t+1}^{t+1})^{1-\alpha}.
\]
Taking the sum of $\|u_*\|_{2t}^p$, we get
\[
\|u_+\|_{2t}^p + \|u_-\|_{2t}^p \leq A_{n,t}^p\lt[(\|\na u_+\|_2^2)^{\alpha} (\|u_+\|_{t+1}^{t+1})^{1-\alpha} +(\|\na u_-\|_2^2)^{\alpha} (\|u_-\|_{t+1}^{t+1})^{1-\alpha}\rt].
\]
Notice that
\[
\|\na u_+\|_2^2 + \|\na u_-\|_2^2 =\|\na u\|_2^2, \quad \text{and}\quad \|u_+\|_{t+1}^{t+1} + \|u_-\|_{t+1}^{t+1} = \|u\|_{t+1}^{t+1},
\]
and by the convexity, we have
\[
a^\alpha b^{1-\alpha} + (A-a)^{\alpha} (B-b)^{1-\alpha} \leq A^{\alpha} B^{1-\alpha},
\]
for any $0\leq a\leq A$ and $0\leq b\leq B$. These observations yield
\begin{align*}
\|u_+\|_{2t}^p + \|u_-\|_{2t}^p &\leq A_{n,t}^p \|\na u\|_2^{2\alpha} \|u\|_{t+1}^{(t+1)(1-\alpha)}= (A_{n,t} \|\na u\|_2^{\theta} \|u\|_{t+1}^{1-\theta})^p= \|u\|_{2t}^p(1+ \deGNS[u])^p.
\end{align*}
Dividing both sides by $\|u\|_{2t}^p$, we get
\begin{equation}\label{eq:tachtichphan}
\lt[\lt(\frac1{\|u\|_{2t}^{2t}} \int_{\R^n} u_+^{2t} dx\rt)^{\frac p{2t}} + \lt(\frac1{\|u\|_{2t}^{2t}} \int_{\R^n} u_-^{2t} dx\rt)^{\frac p{2t}}\rt]^{\frac1p} -1 \leq \deGNS[u].
\end{equation}
Since $p < 2t$, the function $\varphi(a)= (a^{p/2t} +(1-a)^{p/2t})^{1/p} -1$ is concave on $[0,1]$ and $\varphi(a) > 0$ for any $a \in (0,1)$. We claim that
\begin{equation}\label{eq:claim}
\varphi(a) \geq \kappa \min\{a^{\frac p{2t}}, (1-a)^{\frac p{2t}}\},
\end{equation}
for some positive constant $\kappa$ depending only on $n$ and $t$. Indeed, since $\vphi(a) =\vphi(1-a)$, it is enough to check \eqref{eq:claim} for $0< a< 1/2$. Moreover, $\vphi$ is continuous and strict positive on $(0, 1/2)$, hence it suffices to prove \eqref{eq:claim} for $a$ near $0$. Differentiating $\vphi$, we get
\[
\vphi'(a) = \frac{1}{2t} \lt(a^{\frac p{2t}} + (1-a)^{\frac p{2t}}\rt)^{\frac1p -1} \lt(a^{\frac p{2t} -1} -(1-a)^{\frac p{2t} -1}\rt).
\]
Thus for $0 < a < 1/3$, there exists $c >0$ depending on $p$ and$ s$ (thus, on $n$ and $t$) such that $\vphi'(a) \geq c a^{p/(2t) -1}$ hence 
\[
\vphi(a) \geq  \frac{2t c}p a^{\frac p{2t}},\qquad 0< a< 1/3.
\]
This proves our claim \eqref{eq:claim}. Our claim \eqref{eq:claim} and \eqref{eq:tachtichphan} prove \eqref{eq:onepartsmall}.

Without loss of generality, we suppose that the minimum in \eqref{eq:onepartsmall} is attained by $u_-$. Notice that $u = |u| -2 u_-$ and $\deGNS[|u|] \leq \deGNS[u]$. The estimate for nonnegative function asserts that
\begin{equation}\label{eq:1f}
\int_{\R^n} \lt|\frac{\|v\|_{2t}}{\|u\|_{2t}}\, |u| - a^{\frac{n}{2t}} v_{a,x_0}(x)\rt|^{2t} dx \leq  \lt(C\deGNS[|u|]\rt)^{\frac{2t}p} \leq \lt(C\deGNS[u]\rt)^{\frac{2t}p}.
\end{equation}
for some $a >0$ and $x_0 \in \R^n$. Since $u = |u|-2u_-$, then
\begin{align}\label{eq:2ndf}
\int_{\R^n} \Bigg|\frac{\|v\|_{2t}}{\|u\|_{2t}}\, u &- a^{\frac{n}{2t}} v_{a,x_0}(x)\Bigg|^{2t} dx\notag\\
&= \int_{\R^n} \lt|\frac{\|v\|_{2t}}{\|u\|_{2t}}\, (|u|-2u_-) - a^{\frac{n}{2t}} v_{a,x_0}(x)\rt|^{2t} dx\notag\\
&\leq 2^{2t-1}\int_{\R^n} \lt|\frac{\|v\|_{2t}}{\|u\|_{2t}}\, |u| - a^{\frac{n}{2t}} v_{a,x_0}(x)\rt|^{2t} dx + \frac{2^{4t-1}\|v\|_{2t}^{2t}}{\|u\|_{2t}^{2t}} \int_{\R^n} u_-^{2t} dx.
\end{align}
Plugging \eqref{eq:onepartsmall} (with remark that the minimum is taken by $u_-$) and \eqref{eq:1f} into \eqref{eq:2ndf} implies \eqref{eq:quantiativeform}.
\end{proof}

\section{Proof of Theorem \ref{2ndTheorem}}
We follow the argument in the proof of Theorem $1.4$ in \cite{Carlen13}. However, our proof below is simpler with the help of Theorem \ref{1stTheorem}. In our proof, we alwas use $C$ to denote a positive constant depending only on $n,t,p,A$ and $B$, and which value can be changed from lines to lines.

\begin{proof}[Proof of Theorem \ref{2ndTheorem}]
We divide our proof in several steps.\\

\emph{$\bullet$ Step 1: We show that $\|u\|_{2t}$ cannot be too small if $N_p(u)$ is not too large.} Indeed, for any $R>0$ we have
\[
\int_{B_R}|u|^{t+1} dx = \|u\|_{t+1}^{t+1} - \int_{\{|x|>R\}} |u|^{t+1} dx \geq \|v\|_{t+1}^{t+1} - R^{-p} N_p(u).
\]
Choosing $R >0$ such that $R^{-p}N_p(u) = \|v\|_{t+1}^{t+1}/2$ and using H\"older inequality, we get
\[
\frac12 \|v\|_{t+1}^{t+1} \leq \int_{B_R} |u|^{t+1} dx \leq \|u\|_{2t}^{t+1} |B_R|^{\frac{t-1}{2t}} = \|u\|_{2t}^{t+1} (|B_1|\, R^n)^{\frac{t-1}{2t}},
\]
that is 
\begin{equation}\label{eq:lowerbound2tnorm}
\|u\|_{2t}^{2t} \geq c_1 N_p(u)^{-\frac{n(t-1)}{p(t+1)}}\geq c_1 B^{-\frac{n(t-1)}{p(t+1)}},
\end{equation}
with $c_1$ depends only on $n$ and $t$.\\

\emph{$\bullet$ Step 2: Modifying $u$ by multiple and rescale which do not seriously affect the size of deficit $\hatde[u]$.} Define
\[
\tilde u(x) = \frac{\|v\|_{2t}}{\|u\|_{2t}} a^{\frac n{2t}} u(a x),
\]
where $a >0$ is choosen such that
\[
\frac{t^2 -1}{2n} \|\na \tilde u\|_2^2 = \|\tilde u\|_{t+1}^{t+1}.
\]
Note that $\|\tilde u\|_{2t} = \|v\|_{2t}$, and 
\begin{equation}\label{eq:deficitutilde}
\hatde[\tilde u] =\lt(\frac{\|v\|_{2t}}{\|u\|_{2t}}\rt)^{\frac{4t}{2(t)}}\, \hatde[u].
\end{equation}
By \emph{Step 1} (or \eqref{eq:lowerbound2tnorm}), there is a constant $C >0$ such that
\begin{equation}\label{eq:sosanhdelta}
\hatde[\tilde u] \leq C \hatde[u].
\end{equation}

\emph{$\bullet$ Step 3: Application of Theorem \ref{1stTheorem}.} We first claim that
\begin{equation}\label{eq:claim*}
|\|\tilde u\|_{t+1}^{t+1} -\|v\|_{t+1}^{t+1}| \leq C \hatde[u],
\end{equation}
with $C$ depends on $n,t, p$ and $B$. Indeed, by the definition of $\tilde u$, we have
\[
\frac{t^2 -1}{2n} \|\na \tilde u\|_2^2 = \|\tilde u\|_{t+1}^{t+1},
\]
thus
\[
\hatde[\tilde u] = A_{n,t}^{\frac{4t}{2(t)}} \lt(\frac{2n}{t^2-1}\rt)^{\frac{4\theta t}{2(t)}} \lt|\|\tilde u\|_{t+1}^{t+1} -\|v\|_{t+1}^{t+1} \rt|.
\]
This and \eqref{eq:sosanhdelta} imply the claim \eqref{eq:claim*}. Observe that
\[
\|v\|_{t+1}^{t+1} = \|u\|_{t+1}^{t+1} = a^{\frac{n(t-1)}{2t}}\, \frac{\|u\|_{2t}^{t+1}}{\|v\|_{2t}^{t+1}}\, \|\tilde u\|_{t+1}^{t+1}.
\]
Combining this equality and \eqref{eq:claim*}, there exists $\de' >0$ depending on $n,t, p$ and $B$ such that whenever $\hatde[u]\leq \de'$, we have
\begin{equation}\label{eq:lamu2t}
\lt|\lambda^{\frac{n(t-1)}{2t}} \|u\|_{2t}^{t+1} -\|v\|_{2t}^{t+1}\rt|\leq C \hatde[u].
\end{equation} 

By Theorem \ref{1stTheorem}, we can find $x_0\in \R^n$ such that the translation $\hat u(x) = \tilde u(x-x_0)$ of $\tilde u$ satisfies
\begin{equation*}
\int_{\R^n} \lt|\hat u^{\frac{t+1}2} - v^{\frac{t+1}2}\rt|^2 dx \leq K \hatde[\tilde u] \leq C\hatde[u].
\end{equation*}
Hence, by H\"older inequality, we have
\begin{align}\label{eq:translate}
\int_{\R^n} |\hat u^{t+1} -v^{t+1}| dx &= \lt(\int_{\R^n} \lt|\hat u^{\frac{t+1}2} - v^{\frac{t+1}2}\rt|^2 dx\rt)^{\frac12} \lt(\int_{\R^n} \lt|\hat u^{\frac{t+1}2} + v^{\frac{t+1}2}\rt|^2 dx\rt)^{\frac12}\notag\\
&\leq \lt(\|\tilde u\|_{t+1}^{\frac{t+1}2} + \|\tilde v\|_{t+1}^{\frac{t+1}2}\rt) \lt(\int_{\R^n} \lt|\hat u^{\frac{t+1}2} - v^{\frac{t+1}2}\rt|^2 dx\rt)^{\frac12}\notag\\
&\leq C \hatde[u]^{\frac12},
\end{align}
here we use \eqref{eq:claim*} to bound $\|\tilde u\|_{t+1}$ from above when $\hatde[u] \leq \de'$.\\

\emph{$\bullet$ Step 4: Eliminating the normalization.} Setting $u_a(x) = a^{n/(t+1)} u(a x)$. Then we have $\|u_a\|_{t+1} = \|u\|_{t+1}$ and 
\[
\int_{\R^n} \lt|\tilde u^{t+1} -u_a^{t+1}\rt| dx = \frac{|\|u\|_{2t}^{t+1} a^{\frac{n(t-1)}{2t}} -\|v\|_{2t}^{t+1}|}{\|u\|_{2t}^{t+1} a^{\frac{n(t-1)}{2t}}} \|v\|_{t+1}^{t+1}.
\]
From \eqref{eq:lamu2t}, we see that $\|u\|_{2t}^{t+1} a^{\frac{n(t-1)}{2t}} \geq C$. Therefore, again by \eqref{eq:lamu2t},
\begin{equation}\label{eq:1111}
\|\tilde u^{t+1} -u_a^{t+1}\|_1 \leq C \hatde[u].
\end{equation}
Combining \eqref{eq:1111} with \eqref{eq:translate} gives
\begin{equation}\label{eq:2222}
\|u_a(\cdot -x_0)^{t+1}-v^{t+1}\|_1 \leq C \hatde[u]^{\frac12}.
\end{equation}

\emph{$\bullet$ Step 5: Finding a lower bound for $a$:} Notice that from \eqref{eq:lowerbound2tnorm} and \eqref{eq:lamu2t}, we obtain an upper bound depending only on $n,t,p$ and $B$ for $a$ when $\hatde[u]$ smaller than some constant depending only on $n,t,p$ and $B$. In this step, we use the bound on entropy \eqref{eq:bounds} to give a lower bound for $a$.

It follows from \eqref{eq:2222} that 
\[
\int_{B_a(a x_0)} u(x)^{t+1} dx = \int_{B_1} u_a(x-x_0)^{t+1} dx \geq \frac12 \int_{B_1} v(x)^{t+1} dx \geq C.
\] 
Hence by Jensen's inequality, we have
\begin{align*}
&\frac{1}{a^n |B_1|} \int_{B_a(a x_0)} u(x)^{t+1}\, \ln (u(x)^{t+1}) dx\\
&\qquad \geq \frac{1}{a^n |B_1|} \int_{B_a(a x_0)} u(x)^{t+1} dx \, \ln\lt(\frac{1}{a^n |B_1|} \int_{B_a(a x_0)} u(x)^{t+1} dx\rt)\\
&\qquad \geq \frac{C}{a^n} \, \ln\lt(\frac{C}{a^n}\rt).
\end{align*}
Since $b \ln b \geq -1/e$ for any $b >0$, thus we get
\begin{equation}\label{eq:xc}
\int_{B_a(a x_0)} u(x)^{t+1}\, \ln (u(x)^{t+1}) dx \geq -C\lt(1+ \ln a\rt).
\end{equation}
For any nonnegative integrable function $\rho$ on $\R^n$ with finite $p-$moment for some $p\geq 1$, we have the following standard estimate
\[
\int_{\R^n} \rho(x) \, \ln_- \rho(x) dx \leq \int_{\R^n} |x|^p \rho(x) dx + \frac1e\int_{\R^n} e^{-|x|^p}dx,
\]
where $\ln_- s = \max\{-\ln s, 0\}$. Indeed, by writing $\rho(x) = f(x) e^{-|x|^p}$, we have 
\[
\ln_- \rho(x) = \max\{|x|^p -\ln f(x), 0\} \leq |x|^p + \ln_-(f(x)),
\]
hence
\begin{align*}
\int_{\R^n} \rho(x) \, \ln_- \rho(x) dx &\leq \int_{\R^n} |x|^p \rho(x) dx + \int_{\R^n} f(x) \ln_- f(x) e^{-|x|^p} dx \\
& \leq \int_{\R^n} |x|^p \rho(x) dx + \frac1e\int_{\R^n} e^{-|x|^p}dx,
\end{align*}
since $ b \,\ln_- b \leq 1/e$ for any $b>0$. Applying this estimate for $u(x)^{t+1}$
\[
\int_{\R^n} u(x)(x)^{t+1} \, \ln_- (u(x)^{t+1}) dx \leq C,
\]
and hence
\begin{equation}\label{eq:cx}
\int_{\{u(x)\geq 1\}} u(x)^{t+1} \ln (u(x)^{t+1}) dx \leq S(u)+ C.
\end{equation}
From \eqref{eq:xc}, \eqref{eq:cx} and \eqref{eq:bounds}, we obtain $-\ln a \leq C$ for some constant $C$ depending on $n,t,p,B$ and A. This gives us a lower bound for $a$.\\

\emph{Step 6: Reabsorbing $x_0$.} By \eqref{eq:centercond}, we then have $\int_{\R^n} x\, u_a(x)^{t+1} dx =0$, and hence
\[
x_0\|v\|_{t+1}^{t+1} =x_0 \|u\|_{t+1}^{t+1} = \int_{\R^n} x\, u_a(x-x_0)^{t+1} dx =\int_{\R^n} x\,\lt( u_a(x-x_0)^{t+1}-v(x)^{t+1}\rt) dx.
\]
It follows from H\"older inequality and \eqref{eq:2222} that
\begin{align*}
|x_0|\|v\|_{t+1}^{t+1}&\leq \|u_a(\cdot -x_0)^{t+1} -v^{t+1}\|_1^{\frac{p-1}{p}} \lt(N_p(u_a(\cdot-x_0))^{\frac1p} + N_p(v)^{\frac1p}\rt)\\
&\leq C \hatde[u]^{\frac{p-1}{2p}}\lt(N_p(v) + |x_0| \|u\|_{t+1}^{\frac{t+1}p} + a^{-1} N_p(u)^{\frac1p}\rt).
\end{align*}
Therefore, $|x_0|\leq C \hatde[u]^{(p-1)/(2p)}$ whenever $\hatde[u]$ small since $a$ is bounded from below as in \emph{Step 5}.

By the fundamental theorem of calculus and H\"older inequality,
\[
\|\hat u^{t+1} -\tilde u^{t+1}\|_1 = \|\tilde u^{t+1} -\tilde u(\cdot-x_0)^{t+1}\|_1 \leq (t+1) |x_0| \|\na \tilde u\|_2 \|\tilde u\|_{2t}^t.
\]
Since $\tilde u$ satisfies \eqref{eq:conditiononu}, and $\hatde[\tilde u] \leq C \hatde[u]$, then $\|\na \tilde u\|_2 \leq C$. Notice that $\|\tilde u\|_{2t} = \|v\|_{2t}$, hence using the bound on $x_0$ above, we obtain
\[
\|\hat u^{t+1} -\tilde u^{t+1}\|_1 \leq C \hatde[u]^{\frac{p-1}{2p}}.
\]
Combining this estimate and \eqref{eq:translate} gives
\[
\|\tilde u^{t+1} -v^{t+1}\|_1 \leq C \hatde[u]^{\frac{p-1}{2p}}.
\]
Finally, this estimate and \eqref{eq:1111} imply
\[
\|u_a^{t+1} -v^{t+1}\|_1 \leq C \hatde[u]^{\frac{p-1}{2p}},
\]
which is equivalent to \eqref{eq:2ndstability}.

\end{proof}

\section*{Acknowledgments}
This work was supported by the CIMI's postdoctoral research fellowship.


\begin{thebibliography}{99}
\bibitem{Aubin}
T. Aubin, \emph{Probl\`emes isop\'erim\'etriques et espaces de Sobolev\text}, J. Differential Geom., {\bf 11} (1976) 573--598.

\bibitem{Bakry}
D. Bakry, I. Gentil, and M. Ledoux, \emph{Analysis and geometry of Markov diffusion operators\text}, Grundlehren der Mathematischen Wissenschaften (Fundamental Principles of Mathematical Sciences), vol. 348, Springer, Cham, 2014, xx+552 pp.

\bibitem{Barchiesi}
M. Barchiesi, A. Brancolini, and V. Julin, \emph{Sharp dimension free quantitative estimates for the Gaussian isoperimetric inequality\text}, to appear in Ann. Probab.,

\bibitem{Bianchi}
G. Bianchi, and H. Egnell, \emph{A note on the Sobolev inequality\text}, J. Funct. Anal., {\bf 100} (1991) 18--24.

\bibitem{Bobkov}
S. G. Bobkov, N. Gozlan, C. Roberto, and P. M. Samson, \emph{Bounds on the deficit in the logarithmic Sobolev inequality\text}, J. Funct. Anal., {\bf 267} (2014) 4110--4138.

%\bibitem{Blanchet}
%A. Blanchet, E. A. Carlen, and J. A. Carrillo, \emph{Functional inequalities, thick tails and asymptotics for the critical mass Patlak--Keller--Segel model\text}, J. Funct. Anal., {\bf 262} (2012) 2142--2230.

\bibitem{Bonnesen}
T. Bonnesen, \emph{\"Uber das isoperimetrische defizit ebener figuren (German)\text}, Math. Ann., {\bf 91} (1924) 252--268.

%\bibitem{Bonforte}
%M. Bonforte, and J. L. V\'azquez, \emph{Global positivity estimates and Harnack inequalities for the fast diffusion equation\text}, J. Funct. Anal., {\bf 240} (2006) 399--428.

\bibitem{Brezis}
H. Brezis, and E. H. Lieb, \emph{Sobolev inequalities with remainder terms\text}, J. Funct. Anal., {\bf 62} (1985) 73--86.

\bibitem{Brothers}
J. E. Brothers, and W. P. Ziemer, \emph{Minimal rearrangements of Sobolev functions\text}, J. Reine. Angew. Math., {\bf 348} (1988) 153--179.

%\bibitem{Caffarelli}
%L. Caffarelli, R. Kohn, and L. Nirenberg, \emph{First order interpolation inequalities with weights\text}, Compositio Math., {\bf 53} (1984) 259--275.

\bibitem{CarlenLoss}
E. A. Carlen, and M. Loss, \emph{Sharp constant in Nash's inequality\text}, Int. Math. Res. Not. (IMRN), (1993) 213--215.

%\bibitem{Carlen10}
%E. A. Carlen, J. A. Carrilo, and M. Loss, \emph{Hardy--Littlewood--Sobolev inequalities via fast diffusion flows\text}, Proc. Natl. Acad. Sci. USA, {\bf 107} (2010) 19696--19701.

\bibitem{Carlen13}
E. A. Carlen, and A. Figalli, \emph{Stability for a GNS inequality and the log--HLS inequality, with application to the critical mass Keller--Segel equation\text}, Duke Math. J., {\bf 162} (2013) 579--625.

\bibitem{Carlen14}
E. A. Carlen, R. L. Frank, and E. H. Lieb, \emph{Stability estimates for the lowest eigenvalue of a Schrödinger operator\text}, Geom. Funct. Anal., {\bf 24} (2014) 63--84. 

\bibitem{Carlen16}
E. A. Carlen, \emph{Duality and stability for functional inequalities\text}, to appear in The Annales de la Facult\'e des Sciences de Toulouse.

\bibitem{Chen}
S. Chen, R. L. Frank, and T. Weth, \emph{Remainder terms in the fractional Sobolev inequality\text}, Indiana Univ. Math. J., {\bf 62} (2013) 1381--1397.

\bibitem{CianchiBV}
A. Cianchi, \emph{A quantitative Sobolev inequality in BV\text}, J. Funct. Anal., {\bf 237} (2006) 466--481.

\bibitem{CianchiPS}
A. Cianchi, L. Esposito, N. Fusco, and C. Trombetti, \emph{A quantitative P\'olya--Szeg\"o principle\text}, J. Reine. Angew. Math., {\bf 614} (2008) 153--189.

\bibitem{Cianchi}
A. Cianchi, N. Fusco, F. Maggi, and A. Pratelli, \emph{The sharp Sobolev inequality in quantitative form\text}, J. Eur. Math. Soc., {\bf 11} (2009) 1105--1139.

\bibitem{Cianchigauss}
A. Cianchi, N. Fusco, F. Maggi, and A. Pratelli, \emph{On the isoperimetric deficit in the Gauss space\text}, Amer. J. Math., {\bf 133} (2011) 131--186.

%\bibitem{Ciraolo}
%G. Ciraolo, A. Figalli, and F. Maggi, \emph{A quantitative analysis of metrics on $\R^n$ with almost constant positive scalar curvature, with applications to fast diffusion flows\text}, preprint, arXiv:1602.01916v3.

\bibitem{Cordero}
D. Cordero-Erausquin, B. Nazaret, and C. Villani, \emph{A mass-transportation approach to sharp Sobolev and Gagliardo--Nirenberg inequalities\text}, Adv. Math., {\bf 182} (2004) 307--332.

\bibitem{DelPino}
M. Del Pino, and J. Dolbeault, \emph{Best constants for Gagliardo--Nirenberg inequalities and applications to nonlinear diffusions\text}, J. Math. Pures Appl., {\bf 81} (2002) 847--875.

\bibitem{DelPino03}
M. Del Pino, and J. Dolbeault, \emph{The optimal Euclidean $L_p-$Sobolev logarithmic inequality\text}, J. Funct. Anal., {\bf 197} (2003) 151--161.

%\bibitem{Dolbeault04}
%J. Dolbeault, and B. Perthame, \emph{Optimal critical mass in the two--dimensional Keller--Segel model in $\mathbb R^2$\text}, C. R. Math. Acad. Sci. Paris, {\bf 339} (2004) 611--616.

\bibitem{Dolbeault11}
J. Dolbeault, \emph{Sobolev and Hardy-Littlewood-Sobolev inequalities: duality and fast diffusion\text}, Math. Res. Lett., {\bf 18} (2011) 1037--1050.

\bibitem{Dolbeault13}
J. Dolbeault, and G. Toscani, \emph{Improved interpolation inequalities, relative entropy and fast diffusion equations\text}, Ann. Inst. H. Poincar\'e Anal. Non Lin\'eaire, {\bf 30} (2013) 917--934.

\bibitem{Dolbeault14}
J. Dolbeault, and G. Jankowiak, \emph{Sobolev and Hardy--Littlewood--Sobolev inequalities\text}, J. Differential Equations, {\bf 257} (2014) 1689--1720.

\bibitem{Dolbeault16}
J. Dolbeault, and G. Toscani, \emph{Stability results for logarithmic Sobolev and Gagliardo-Nirenberg inequalities\text}, Int. Math. Res. Not. (IMRN), (2016) 473--498.

\bibitem{Eldan}
R. Eldan, \emph{A two--sided estimate for the Gaussian noise stability deficit\text}, Invent. Math., {\bf 201} (2015) 561--624.

\bibitem{Fathi}
M. Fathi, E. Indrei, and M. Ledoux, \emph{Quantitative logarithmic Sobolev inequalities and stability estimates\text}, Discrete Contin. Dyn. Syst., {\bf 36} (2016) 6835--6853.

\bibitem{FigalliBMconvex}
A. Figalli, F. Maggi, and A. Pratelli, \emph{A refined Brunn--Minkowski inequality for convex sets\text}, Ann. Inst. Henri Poincar\'e Anal. Non Lin\'eaire, {\bf 26} (2009) 2511--2519.

\bibitem{Figalliiso}
A. Figalli, F. Maggi, and A. Pratelli, \emph{A mass transportation approach to quantitative isoperimetric inequalities\text}, Invent. Math., {\bf 182} (2010) 167--211.

\bibitem{FigalliBV}
A. Figalli, F. Maggi, and A. Pratelli, \emph{Sharp stability theorems for the anisotropic Sobolev and log-Sobolev inequality on functions of bounded variation\text}, Adv. Math., {\bf 242} (2013) 80--101. 

\bibitem{Figallisumset}
A. Figalli, and D. Jerison, \emph{Quantitative stability for sumsets in $\R^n$\text}, J. Eur. Math. Soc., (JEMS), {\bf 17} (2015) 1079--1106.

\bibitem{FigalliBM}
A. Figalli, and D. Jerison, \emph{Quantitative stability for the Brunn--Minkowski inequality\text}, preprint, arXiv:1502.06513v1.

\bibitem{Figalli}
A. Figalli, and R. Neumayer, \emph{Gradient stability for the Sobolev inequality: the case $p\geq 2$\text}, to appear in Journal of European Mathematical Society (JEMS).


\bibitem{Fuglede}
B. Fuglede, \emph{Stability in the isoperimetric problem for convex or nearly spherical domains in $\R^n$\text}, Trans. Amer. Math. Soc., {\bf 314} (1989) 619--638.

\bibitem{Fusco07}
N. Fusco, F. Maggi, and A. Pratelli, \emph{The sharp quantitative Sobolev inequality for function of bounded variation\text}, J. Funct. Anal, {\bf 244} (2007) 315--341.

\bibitem{Fusco08}
N. Fusco, F. Maggi, and A. Pratelli, \emph{The sharp quantitative isoperimetric inequality\text}, Ann. of Math., {\bf 168} (2008) 941--980.

\bibitem{Fusco09}
N. Fusco, F. Maggi, and A. Pratelli, \emph{Stability estimates for certain Faber--Krahn, isocapacitary and Cheeger inequalities\text}, Ann. Sc. Norm. Super. Pisa Cl. Sci., {\bf 8} (2009) 51--71.

\bibitem{Gross}
L. Gross, \emph{Logarithmic Sobolev inequality\text} Amer. J. Math., {\bf 97} (1975) 1061--1083.

\bibitem{Hall91}
R. R. Hall, W. K. Hayman, and A. W. Weitsman, \emph{On asymmetry and capacity\text}, J. Analyse Math., {\bf 56} (1991) 87--123.

\bibitem{Hall}
R. R. Hall, \emph{A quantitative isoperimetric inequality in $n-$dimensional space\text}, J. Reine. Angew. Math., {\bf 428} (1992) 161--176.

\bibitem{Indrei}
E. Indrei, and D. Marcon, \emph{A quantitative log--Sobolev inequality for a two parameter family of functions\text}, Int. Math. Res. Not. (IMRN), (2014) 5563--5580.

\bibitem{Jankowiak}
G. Jankowiak, and V. H. Nguyen, \emph{Fractional Sobolev and Hardy--Littlewood--Sobolev inequalities\text}, revision in Adv. Differential Equations. 

\bibitem{Loiudice}
A. Loiudice, \emph{Improved Sobolev inequalities on the Heisenberg group\text}, Nonlinear Anal., {\bf 62} (2005) 953--962.

\bibitem{Lu}
G. Lu, and J. Wei, \emph{On a Sobolev inequality with remainder term\text}, Proc. Amer. Math. Soc., {\bf 128} (2000) 75--84.

%\bibitem{Maz'ya}
%V. Maz'ya, \emph{Sobolev Spaces: with Applications to Elliptic Partial Differential Equations\text}, second, revised and augmented edition, Springer, Heidelberg (2011).

\bibitem{Maggi}
F. Maggi, \emph{Some methods for studying stability in isoperimetric type problems\text}, Bull. Amer. Math. Soc., {\bf 45} (2008) 367--408.

\bibitem{Mossel}
E. Mossel, and J. Neeman, \emph{Robust dimension free isoperimetry in Gaussian space\text}, Ann. Probab., {\bf 43} (2015) 971--991.  

\bibitem{Nguyen15}
V. H. Nguyen, \emph{Sharp weighted Sobolev and Gagliardo-Nirenberg inequalities on half--spaces via mass transport and consequences\text}, Proc. Lond. Math. Soc. (3), {\bf 111} (2015) 127--148.

\bibitem{Nguyen16}
V. H. Nguyen, \emph{Stability results for Gagliardo-Nirenberg inequality via mass transport\text}, In preparation.

\bibitem{Rosen}
G. Rosen, \emph{Minimum value for $c$ in the Sobolev inequality $\|\phi^3\|\leq c\|\nabla \phi\|^3$\text}, SIAM J. Appl. Math., {\bf 21} (1971) 30--32. 

\bibitem{Ruffini}
B. Ruffini, \emph{Stability theorems for Gagliardo--Nirenberg--Sobolev inequalities: a reduction principle to the radial case\text}, Rev. Mat. Complut., {\bf 27} (2014) 509--539.

\bibitem{Seuffert1}
F. Seuffert, \emph{An extension of the Bianchi-Egnell stability estimate to Bakry, Gentil, and Ledoux's generalization of the Sobolev inequality to continuous dimensions\text}, preprint, arXiv:1512.06121v1.

\bibitem{Seuffert2}
F. Seuffert, \emph{A stability result for a family of sharp Gagliardo-Nirenberg inequalities\text}, preprint, arXiv:1610.06869v1.

\bibitem{Talenti}
G. Talenti, \emph{Best constants in Sobolev inequality\text}, Ann. Mat. Pura Appl., {\bf 110} (1976) 353--372.
\end{thebibliography}
\end{document}